%% file: ct_qms_arXiv3.tex
\documentclass[11pt]{article}
\usepackage[margin=1.2in]{geometry}

\usepackage[T1]{fontenc}

\usepackage{latexsym, amssymb, bbm}
\usepackage{amsthm}
\usepackage{xcolor}
\usepackage{graphicx}
\usepackage{amsmath}
\usepackage{indentfirst}

\numberwithin{equation}{section}

\newtheorem{thm}{Theorem}[section]
\newtheorem{lemma}[thm]{Lemma}
\newtheorem{prop}[thm]{Proposition}
\newtheorem{coroll}[thm]{Corollary}

\theoremstyle{definition}

\newtheorem{defin}[thm]{Definition}
\newtheorem{rem}[thm]{Remark}
\newtheorem{exam}[thm]{Example}

\newcommand{\R}{{\mathbb{R}}}

\newcommand{\T}{{\mathbb{T}}}
\newcommand{\Z}{{\mathbb{Z}}}
\newcommand{\N}{{\mathbb{N}}}
\newcommand{\C}{{\mathbb{C}}}
\newcommand{\K}{{\mathbb{K}}}

\newcommand{\bP}{{\mathbb{P}}}

\newcommand{\cL}{{\mathcal{L}}}

\newcommand{\cN}{{\mathcal{N}}}
\newcommand{\cO}{{\mathcal{O}}}
\newcommand{\cP}{{\mathcal{P}}}

\newcommand{\mft}{{\mathfrak{t}}}
\newcommand{\mfk}{{\mathfrak{k}}}

\newcommand{\abs}[1]{\left| #1 \right|}

\newcommand{\ip}[1]{\langle #1 \rangle}

\newcommand{\w}{\omega}

\def\bK{\mathbb{K}}

\def\g{\gamma}
\def\a{\alpha}

\def\mf{\mathfrak}

\def\CP{\mathbb{C}P}
\def\RP{\mathbb{R}P}

\DeclareMathOperator{\Lie}{Lie}
\DeclareMathOperator{\Giv}{Giv}
\DeclareMathOperator{\std}{std}
\DeclareMathOperator{\Clif}{Clif}
\DeclareMathOperator{\KPS}{KPS}

\newcommand{\fc}{{:\ }}

\newcommand{\ol}{\overline}
\newcommand{\wt}{\widetilde}
\newcommand{\wh}{\widehat}

\newcommand{\tb}{\textbf}

\DeclareMathOperator{\id}{id}

\DeclareMathOperator{\Vol}{Vol}

\DeclareMathOperator{\pr}{pr}

\DeclareMathOperator{\supp}{supp}

\DeclareMathOperator{\Ham}{Ham}

\DeclareMathOperator{\Cont}{Cont}
\DeclareMathOperator{\Diff}{Diff}

\DeclareMathOperator{\Span}{span}

\begin{document}

\title{Quasi-morphisms on contactomorphism groups and contact rigidity}
\author{Matthew Strom Borman\thanks{Partially supported by NSF grants DMS-1006610 and DMS-1304252.}\, and Frol Zapolsky\thanks{
Partially supported by Deutsche Forschungsgemeinschaft grant DFG/CI 45/5-1 while at LMU (Munich).}}

\maketitle

\begin{center}
\textit{Dedicated with gratitude to our teacher Leonid Polterovich on his 50th birthday}
\end{center}

\begin{abstract}
We build homogeneous quasi-morphisms on the universal cover of the contactomorphism group for certain 
prequantizations of monotone symplectic toric manifolds.  This is done using Givental's nonlinear Maslov index and a contact 
reduction technique for quasi-morphisms.  We show how these quasi-morphisms lead to a hierarchy of rigid subsets of contact 
manifolds.  We also show that the nonlinear Maslov index has a vanishing property, which plays a key role in our proofs.
Finally we present applications to orderability of contact manifolds and Sandon-type metrics on contactomorphism groups.
\end{abstract}
%%%%%%%%%%%%%%%%%%%%%%%%%%%%%
%%%%%%%%%%%%%%%%%%%%%%%%%%%%%

%%%%%%%%%%%%%%%%%%%%%%%%%%%%%%%%%%%%%%%%%%%%%%%%%%%%%%%%%%
%%%%%%%%%%%%%%%%%%%%%%%%%%%%%%%%%%%%%%%%%%%%%%%%%%%%%%%%%%
\section{Introduction and results}\label{section_intro}

%%%%%%%%%%%%%%%%%%%%%%%%%%%%%
%%%%%%%%%%%%%%%%%%%%%%%%%%%%%
\subsection{Quasi-morphisms on contactomorphism groups}\label{section_qms}

A \textbf{quasi-morphism} on a group $G$ is a function $\mu \fc G \to \R$ which is a homomorphism up to a bounded error, 
that is there is $D > 0$ such that
\begin{equation}\label{e:Quasimorphism}
\abs{\mu(ab)-\mu(a)-\mu(b)} \leq D \quad\mbox{for all}\quad a,b \in G\,,
\end{equation}
and it is \textbf{homogeneous} if $\mu(a^{k}) = k\mu(a)$ for all $a \in G$ and $k \in \Z$.  
It is straightforward to show that homogeneous quasi-morphisms are conjugation-invariant and restrict to homomorphisms on abelian subgroups. 
See \cite{Bav91L, Cal09S, Kot04W} for background on quasi-morphisms, their connection with bounded cohomology, and their applications to commutator length and other quantitative group-theoretic questions.
For the sake of exposition, in this paper by quasi-morphism we will mean a non-zero, homogeneous quasi-morphism.

The construction and applications of quasi-morphisms on infinite-dimensional groups of symmetries have recently been a popular theme 
of research \cite{Ent04C, FOOO11S, GamGhy04C, Ghy01G,   Ghy07K, Pol06F, She11T, Ush11D}.  One reason is that groups 
of diffeomorphisms are often perfect \cite{Ban97T}, and thus admit no non-zero homomorphisms to $\R$ and so one is 
led to study quasi-morphisms on them instead. When the group has an interesting metric such as the hydrodynamic metric on the group
of volume-preserving diffeomorphisms of a Riemannian manifold \cite{Bra12Q, BraShe13O} or Hofer's metric on the Hamiltonian group 
of a symplectic manifold \cite{EntPol03C, Py06Q}, quasi-morphisms can be used to understand the coarse geometry of these 
groups.  Another reason is that oftentimes quasi-morphisms on the symmetry groups of symplectic and contact manifolds lead to results on the geometry of the underlying manifolds themselves, which is also the case in the present paper.

We will only consider contact manifolds $(V,\xi)$ where $V$ is connected and closed, unless stated otherwise, and $\xi$ is a 
cooriented contact structure.
We will write $(V, \xi, \alpha)$ if we want to specify a choice of a coorienting contact form $\alpha$ such that $\xi = \ker \alpha$.  The
Reeb vector field associated to a contact form $\alpha$ will be denoted $R_{\alpha}$ and is uniquely defined by
$$
	\alpha(R_{\alpha}) = 1 \quad\mbox{and}\quad \iota_{R_{\alpha}}d\alpha = 0\,.
$$
Let $\Cont_0(V, \xi)$ be the identity component of the group of contactomorphisms and denote by
$\wt \Cont_0(V,\xi)$ its universal cover. 

Given a smooth time-dependent function $h \fc [0,1] \times V \to \R$, called a \tb{contact Hamiltonian}, there is a unique time-dependent vector field 
$\{X_{h_t}\}_{t \in [0,1]}$ satisfying 
\begin{equation}\label{e:ContactHam}
\mbox{$\alpha(X_{h_t}) = h_t$ \quad and\quad $d\alpha(X_{h_t},\cdot) = -dh_t + dh_t(R_\alpha)\alpha$ \quad where $h_{t} = h(t, \cdot)$.}
\end{equation}
The vector field $\{X_{h_{t}}\}$ preserves $\xi$ and integrates into a contact isotopy denoted $\{\phi_h^t\}_{t \in [0,1]}$. This establishes 
a bijection, depending on the contact form $\alpha$, between smooth functions $h \fc [0,1] \times V \to \R$ and contact isotopies based at the identity $\id$ of $V$. If $h,g \in C^\infty(V)$, then
\begin{equation}\label{e:ContactPB}
	\{h,g\}_\alpha := -dg(X_h) + dh(R_\alpha)g
\end{equation}
is the contact Hamiltonian corresponding to the Lie bracket of $X_h$ and $X_g$. We write $\wt\phi_h$ for the element of 
$\wt\Cont_{0}(V, \xi)$ represented by the contact isotopy $\{\phi^{t}_h\}_{t \in [0,1]}$.  For the constant function $h=1$, the 
vector field $X_{1} = R_{\alpha}$ is the Reeb vector field and hence $\wt\phi_{1}$ is the element generated by the Reeb flow.

Following Eliashberg--Polterovich \cite{EliPol00P} for $\wt\phi \in \wt\Cont_0(V,\xi)$ we will write $\id \preceq \wt{\phi}$ if there is 
a nonnegative contact Hamiltonian $h$ such that $\wt\phi = \wt\phi_h$ in $\wt\Cont_{0}(V, \xi)$. 
The nonnegativity of $h$ is equivalent to $X_{h_{t}}$ being nowhere negatively transverse to $\xi$, and 
therefore it is independent of $\alpha$. This induces a reflexive and transitive relation on $\wt\Cont_{0}(V, \xi)$ where
\begin{equation}\label{e:PartialOrder}
	\wt\phi \preceq \wt\psi \quad\mbox{if and only if}\quad \id \preceq \wt\phi^{-1}\wt\psi
\end{equation}
which is also bi-invariant \cite{EliPol00P}.  The contact manifold $(V, \xi)$ is called \textbf{orderable} if $\preceq$ is a partial order
on $\wt{\Cont}_{0}(V, \xi)$, that is $\preceq$ is also anti-symmetric.  

\begin{defin}
For a quasi-morphism $\mu\fc \wt\Cont_{0}(V, \xi) \to \R$, define the following properties:
\begin{enumerate}
	\item[(i)] \textbf{Monotone:} $\wt\phi \preceq \wt\psi$ implies $\mu(\wt\phi) \leq \mu(\wt\psi)$.
	\item[(ii)] \textbf{$C^0$-continuous:} 	If $h$ is a smooth contact Hamiltonian
	and there is a sequence of smooth contact Hamiltonians
	$h^{(n)}$ such that $h^{(n)} \to h$ in $C^0([0,1] \times V)$, then $\mu(\wt\phi_{h^{(n)}}) \to \mu(\wt\phi_h)$.
	\item[(iii)] \textbf{Vanishing:} If $U \subset V$ is an open subset and there is $\psi \in \Cont_0(V,\xi)$ with 
	$\psi(U) \cap U = \emptyset$, then $\mu(\wt\phi_{h}) = 0$ for all contact Hamiltonians $h$ with 
	$\supp(h) \subset [0,1] \times U$.
\end{enumerate}
\end{defin}
\noindent
In general, a subset $S \subset V$ is \textbf{displaceable} if there is $\psi \in \Cont_{0}(V, \xi)$ with $\psi(S) \cap \ol{S} = \emptyset$. Note that the vanishing property is independent of the choice of a contact form.

%%%%%%%%%%%%%%%%%%%%%%
\subsubsection{Givental's asymptotic nonlinear Maslov index}

Besides Poincar\'e's rotation number on $\wt\Cont_0(S^1) \equiv \wt{\Diff}_{0}(S^{1})$, the only previous construction of quasi-morphisms on contactomorphism groups was Givental's \textbf{asymptotic nonlinear Maslov index} \cite[Section 9]{Giv90N}
\begin{equation}\label{e:GiventalQM}
	\mu_{\Giv}\fc \wt\Cont_{0}(\RP^{2d-1}) \to \R\,,
\end{equation}
with $\R P^{2d-1}$ being taken with the standard contact structure. 
Results in \cite[Section 9]{Giv90N} imply $\mu_{\Giv}$ is a homogeneous quasi-morphism, as Ben Simon \cite[Theorem 0.2]{Ben07T} proved.
In Section~\ref{s:VanishingProof} we will review the definition and relevant properties of
Givental's quasi-morphism, and prove the following proposition.

\begin{prop}\label{prop_ANLMI_has_vanishing}
	Givental's quasi-morphism $\mu_{\Giv}\fc \wt{\Cont}_0(\RP^{2d-1}) \to \R$ 
	is (i) monotone, (ii) $C^0$-continuous, and (iii) has the vanishing property.
\end{prop}

For time-independent contact Hamiltonians Givental proved \cite[Corollary 3, Section 9]{Giv90N} that $\mu_{\Giv}$ is monotone and 
$C^0$-continuous, and as we will explain his proofs work in general.  The vanishing property, which does not appear in \cite{Giv90N},
together with Theorem~\ref{thm_HamQuasimorphism} below give an alternative proof of Ben Simon's \cite[Theorem 0.6]{Ben07T}.
%%%%%%%%%%%%%%%%%%%%%%

%%%%%%%%%%%%%%%%%%%%%%%%%%%%%
\subsubsection{Quasi-morphisms for prequantizations of even toric manifolds}
  
A \textbf{prequantization} of a symplectic manifold $(M, \w)$ is a contact manifold $(V, \xi, \alpha)$
with a map
$\pi\fc (V, \alpha) \to (M, \w)$ 
defining a principal $S^{1}$-bundle
such that $\pi^{*}\w = d\alpha$, and the Reeb vector
field $R_{\alpha}$ induces the free $S^{1}$-action on $V$ where $S^{1} = \R/\hbar\Z$, $\hbar > 0$ being the minimal period of a closed Reeb orbit.

A \textbf{toric} symplectic manifold $(M^{2n},\omega, \T)$ is a symplectic manifold endowed with an effective Hamiltonian action of a torus $\T$ of dimension $n$.  
The action is induced by a moment map $M \to \mft^*$, where $\mft^*$ is the dual of the Lie algebra $\mft$ of $\T$,
and the image of the moment map is called the moment polytope and denoted $\Delta$. If $\Delta$ has $d$ facets, then it is given by
\begin{equation}\label{e:Polytope}
	\Delta = \{x \in \mf{t}^{*}\,|\, \ip{\nu_{j}, x} + a_{j} \geq 0 \mbox{ for $j=1, \dots, d$}\}\,,
\end{equation}
where the conormals $\nu_j$ are primitive vectors in the integer lattice $\mft_\Z := \ker (\exp \fc\mft \to \T)$. 

A symplectic manifold $(M,\omega)$ is \tb{monotone} if and only if there is a positive constant $\lambda > 0$ so that
$[\w] = \lambda\,c_{1}(M) \in H^{2}(M;\R)$, and for toric manifolds this is equivalent to being able to choose the moment map so that $a_1=\dots=a_d = \lambda $. We call the moment polytope $\Delta$ \tb{even} if $\sum_{j=1}^{d}\nu_j \in 2\mft_\Z$ and we say that
a toric manifold is even if its associated moment polytope is even.  
In Section~\ref{s:EvenExample} we give examples of closed monotone even symplectic toric manifolds. 
We can now formulate our main result.

\begin{thm}\label{thm_main}
	Every closed monotone even toric symplectic manifold $(M, \w, \T)$
	has a prequantization $(\wh{M}, \xi, \alpha)$ for which there is a quasi-morphism 
	$$
	\mu\fc \wt\Cont_{0}(\wh{M}, \xi) \to \R
	$$
	that is monotone, has the vanishing property, and is $C^0$-continuous.
\end{thm}

In Section \ref{s:SymplecticQMQS} below we discuss the significance of this theorem in the context of stable Calabi quasi-morphisms on the universal cover of the Hamiltonian group of a symplectic manifold.

Theorem~\ref{thm_ART} below shows how a monotone quasi-morphism on $\wt\Cont_{0}(V)$ can
induce a monotone quasi-morphism on $\wt\Cont_{0}(\ol{V})$ if $(\ol{V}, \ol{\xi})$ is the result of performing
contact reduction on $(V, \xi)$.
In Section~\ref{s:BuildPrequantization} we will show how the even moment polytope of a monotone toric manifold $(M, \w, \T)$ naturally
leads to a prequantization $(\wh{M}, \xi, \alpha)$ obtained from $\RP^{2d-1}$ via contact reduction.
The proof of Theorem~\ref{thm_main} is then given in Section~\ref{s:ApplyingART} where
we apply Theorem~\ref{thm_ART} to Givental's quasi-morphism $\mu_{\Giv}$ on $\wt\Cont_{0}(\RP^{2d-1})$ to build 
the monotone quasi-morphisms $\mu\fc \wt\Cont_{0}(\wh{M}, \xi) \to \R$.

	Not all prequantizations $\pi \fc(V, \alpha) \to (M, \w)$ of a monotone even toric symplectic manifold $(M, \w)$ admit
	non-trivial monotone quasi-morphisms on $\wt\Cont_0(V)$.
	This is because if $V$ is not orderable, then there is no monotone quasi-morphism on $\wt\Cont_0(V)$ 
	(see Theorem \ref{t:orderability} below).
	The basic example is the standard contact sphere $S^{2d-1}$ for
	$d \geq 2$, which is a prequantization of the even toric manifold $\CP^{d-1}$ but is not orderable.
	See Section~\ref{section_orderability} for further discussion about orderability and quasi-morphisms.  

\begin{rem}If $\pi \fc (V, \alpha) \to (M, \w)$ is a prequantization, then for the subgroup $\Z_{k} \leq S^{1}$ the quotient manifold 
$V/\Z_k$ is also a prequantization of $M$.  Pulling back contact Hamiltonians via the projection $V \to V/\Z_k$ 
induces a homomorphism $\wt\Cont_0(V/\Z_k) \to \wt\Cont_0(V)$
and therefore the quasi-morphisms of Theorem \ref{thm_main} give rise to quasi-morphisms on $\wt\Cont_0(\wh{M}/\Z_k)$.
\end{rem}

\begin{rem}
	There is work in progress by Karshon--Pabiniak--Sandon \cite{KarPabSan13P} to generalize Givental's 
	construction of the asymptotic nonlinear Maslov index, with lens spaces being the first step.
	If for a prime $p$ there is a monotone quasi-morphism with the vanishing property
	$$
		\mu_{\KPS}\fc \wt\Cont_{0}(S^{2d-1}\!/\Z_{p}) \to \R
	$$
	where $\Z_{p}$ acts by multiplication by a $p$-th root of unity, then
	Theorem~\ref{thm_main} would generalize to the
	closed monotone toric symplectic manifolds $(M, \w, \T)$
	whose sum of conormals in the moment polytope satisfies $\sum_{j=1}^{d}\nu_j \in p\cdot\mft_\Z$
	(see the proof of Lemma~\ref{l:Even}).	
\end{rem}
%%%%%%%%%%%%%%%%%%%%%%%%%%%%%

%%%%%%%%%%%%%%%%%%%%%%%%%%%%%
\subsubsection{Reduction for quasi-morphisms on contactomorphism groups}\label{section_reduction_method_qms}

In \cite{Bor12S, Bor13Q} a procedure for pushing forward quasi-morphisms on the universal cover of the Hamiltonian group of a symplectic manifold via symplectic reduction was developed.
In this paper we will streamline this technique and adapt it to the contact setting in Theorem~\ref{thm_ART}, which will be used to prove 
Theorem~\ref{thm_main}.   Before we can formulate the reduction theorem for quasi-morphisms, we need the following two definitions.

\begin{defin}\label{d:StrictlyCoiso}
	For a contact manifold $(V, \xi, \alpha)$, a closed submanifold $Y \subset V$ transverse to $\xi$ is
	\textbf{strictly coisotropic with respect to $\alpha$}
	if it is coisotropic, that is the subbundle $TY \cap \xi$ of the symplectic vector bundle $(\xi|_Y,d\alpha)$ is coisotropic:\
	\begin{equation}\label{e:Coisotropic}
		\{ X \in \xi_{y} \mid \iota_{X} d\alpha = 0 \mbox{ on } T_{y}Y \cap \xi_{y}\} \subset T_{y}Y \cap \xi_{y}
		\quad\mbox{for all $y \in Y$},
	\end{equation}
	and additionally $R_{\alpha}(y) \in T_{y}Y$ for all $y \in Y$, that is\ the Reeb vector field is tangent to $Y$.
\end{defin}
\noindent
The property of being coisotropic is independent of the contact form and, assuming transversality, being strictly co\-iso\-tropic with respect to $\alpha$ is equivalent to
\begin{equation}\label{e:StrictCoisotropic}
	T_{y}Y^{d\alpha}:= \{ X \in T_{y}V \mid \iota_{X} d\alpha = 0 \mbox{ on } T_{y}Y\} \subset T_{y}Y
	\quad\mbox{for all $y \in Y$}.
\end{equation}
One can check that $Y \subset (V, \xi)$ is strictly coisotropic with respect to some contact form if and only if $Y$ is the diffeomorphic
image of a coisotropic submanifold under the projection $SV \to V$ where $SV$ is the symplectization of $V$.

\begin{defin}\label{d:subheavy}
Let $\mu \fc \wt\Cont_0(V,\xi) \to \R$ be a monotone quasi-morphism. A closed subset\footnote{See Remark \ref{r:rigidSetsClosed} regarding the closed assumption, which also applies to the definition of superheavy below.} $Y \subset V$ is $\mu$-\tb{subheavy}
if
$$\mu(\wt\phi_h) = 0$$
whenever $h$ is an autonomous contact Hamiltonian with $h|_{Y} = 0$.
\end{defin}

\noindent
Here now is the reduction theorem for quasi-morphisms on contactomorphism groups, which we will prove in Section~\ref{s:ProofOfART}.
Consider the setting
\begin{equation}\label{e:SetUpART}
	(V, \xi, \alpha) \supset (Y, \alpha|_{Y}) \stackrel{\rho}{\longrightarrow} (\ol{V}, \ol{\xi}, \ol{\alpha})
\end{equation}
where $(V, \xi, \alpha)$ and $(\ol{V}, \ol{\xi}, \ol{\alpha})$ are closed contact manifolds,
$Y \subset V$ is a closed submanifold that is strictly coisotropic with respect to $\alpha$,
and $\rho\fc Y \to \ol{V}$ is a fiber bundle such that $\rho^{*}\ol{\alpha} = \alpha|_{Y}$.

\begin{thm}\label{thm_ART}
	In the setting \eqref{e:SetUpART} if $Y \subset V$ is subheavy for a monotone quasi-morphism
	$\mu \fc \wt\Cont_{0}(V, \xi) \to \R$, then it induces
	a monotone quasi-morphism
	\begin{equation}\label{e:ReducedQM}
	\overline{\mu}\fc \wt\Cont_{0}(\overline{V}, \overline{\xi}) \to \R
	\quad\mbox{defined by}\quad
	\overline{\mu}(\wt{\phi}_{\overline{h}}) := \mu(\wt{\phi}_{h})
	\end{equation}
	where $h \in C^{\infty}([0,1] \times V)$ is any contact Hamiltonian such that
	$h|_{[0,1]\times Y} = \rho^{*}\overline{h}$.
	The vanishing property and $C^0$-continuity passes from $\mu$ to $\overline{\mu}$.
\end{thm}

An example of \eqref{e:SetUpART} is given by contact reduction \cite[Theorem 6]{Gei97C} where a compact Lie group 
$G$ acts on $V$ preserving $\alpha$ with moment map $P\fc V \to \mathfrak{g}^{*}$.  In this case $Y = P^{-1}(0)$ is strictly 
coisotropic with respect to $\alpha$ and $\ol{V} = Y/G$ is a contact manifold assuming $G$ acts freely on $Y$.  
When we prove Theorem~\ref{thm_main} in Section~\ref{s:ApplyingART} it will be in the case of contact reduction
for torus actions on $\RP^{2d-1}$.

It should be noted that, considering more general group actions on $\R P^{2d-1}$, it is possible to construct monotone quasi-morphisms with the vanishing property on prequantizations of symplectic manifolds more general than toric ones, however we shall not pursue this direction here.
%%%%%%%%%%%%%%%%%%%%%%%%%%%%%

%%%%%%%%%%%%%%%%%%%%%%%%%%%%%
%%%%%%%%%%%%%%%%%%%%%%%%%%%%%

%%%%%%%%%%%%%%%%%%%%%%%%%%%%%
%%%%%%%%%%%%%%%%%%%%%%%%%%%%%
\subsection{Contact rigidity}\label{s:ContactRigidity}
 
Nondisplaceability phenomena in contact manifolds is one aspect of contact rigidity and it is
much less studied than nondisplaceability in symplectic manifolds by Hamiltonian diffeomorphisms 
\cite{AbrBorMcD12D,AbrMac13R, BirEntPol04C,Cho04H,EntPol06Q,EntPol09R,FOOO09L,McD11D,WilWoo13Q,Woo11G}.
As with the symplectic setting, contact nondisplaceability goes back to a conjecture of Arnold
that for the standard contact structure on the jet space $J^1N = T^{*}N \times \R$ of a closed
manifold $N$, the zero section $\{(q,0,0)\,|\,q\in N\}$ cannot be displaced from the zero wall 
$\{(q,0,z)\,|\,q\in N,\, z\in\R)\}$ by a contact isotopy and this was proved by Chekanov \cite{Che96C} using generating functions.
Using spectral invariants from generating functions \cite{Zap13G} proved contact rigidity for smooth and
singular subsets of the standard contact $T^*N \times S^1$.
Floer theoretic methods have also been used by Eliashberg--Hofer--Salamon \cite{EliHofSal95L} and Ono \cite{Ono96L} to
detect nondisplaceable submanifolds in unit cotangent bundles of closed manifolds and in certain prequantizations. 
Recently sheaf-theoretic methods have been also been playing a role in symplectic and contact rigidity, see for example \cite{Tamarkin_Microlocal_condition_nondisplaceability, Guillermou_Kashiwara_Schapira_Sheaf_quant_Ham_isotopies_appl_nondisplaceability}. 

In the series of papers \cite{BirEntPol04C, EntPol03C, EntPol06Q, EntPol09R} Entov--Polterovich showed how to use 
the machinery of their quasi-morphisms on the universal cover of the Hamiltonian group of a symplectic manifold $(M, \w)$ 
and \emph{quasi-states} in order to study the rigidity of symplectic intersections.  In particular in
\cite{EntPol09R} they showed that there is a hierarchy of rigid subsets in symplectic manifolds
for which they introduced the terminology of \emph{heavy} and \emph{superheavy} subsets.

%%%%%%%%%%%%%%%%%%%%%%%
\subsubsection{Superheavy and subheavy sets for monotone quasi-morphisms on $\wt\Cont_{0}$}

Inspired by Entov--Polterovich's work, in this paper we will show how monotone quasi-morphisms on 
$\wt\Cont_{0}(V)$ can also be used to study the rigidity of intersections in contact manifolds.
In analogy to the terms \emph{heavy} and \emph{superheavy} for subsets of symplectic manifolds, we will also show 
how such monotone quasi-morphisms detect a hierarchy of rigid subsets in contact manifolds,
namely \emph{subheavy} (defined above) and \emph{superheavy} sets:
\begin{defin}\label{d:sh}
	If $\mu\fc \wt\Cont_{0}(V, \xi) \to \R$ is a monotone quasi-morphism, then a closed subset $Y \subset V$ is
	\textbf{$\mu$-superheavy} if
	$$
	\mu(\wt\phi_{h}) > 0
	$$ for all autonomous contact Hamiltonians $h \in C^{\infty}(V)$ such that $h|_{Y} > 0$.
\end{defin}
Given a prequantization $\pi\fc (V, \alpha) \to (M, \w)$, in Section~\ref{s:SymplecticQMQS} we will discuss how superheavy subsets in
the symplectic manifold $(M, \w)$ are related to subheavy and superheavy subsets of the contact manifold $(V, \alpha)$.
The basic properties of superheavy sets in contact manifolds are given by the following proposition.

\begin{prop}\label{p:Basics}
	Let $\mu \fc \wt\Cont_{0}(V,\xi) \to \R$ be a monotone quasi-morphism.
	\begin{enumerate}
	\item The properties $\mu$-superheavy and $\mu$-subheavy are independent of the choice of contact form $\alpha$ for $\xi$ 
	used to link contact Hamiltonians and contact isotopies.
	\item If $Z$ is $\mu$-superheavy ($\mu$-subheavy) and $Z \subset Y$, then $Y$ is $\mu$-superheavy ($\mu$-subheavy).
	\item The property of being $\mu$-subheavy is preserved by elements of $\Cont_{0}(V, \xi)$, and likewise for $\mu$-superheavy.
	\item The entire manifold $V$ is $\mu$-superheavy.
	\end{enumerate}
\end{prop}
\noindent
This next theorem and its corollary relates subheavy and superheavy sets with contact rigidity.

\begin{thm}\label{t:SHandsubheavy}
Let $\mu \fc \wt\Cont_{0}(V,\xi) \to \R$ be a monotone quasi-morphism.
\begin{enumerate}
\item All $\mu$-superheavy subsets are $\mu$-subheavy.
\item If $Y$ is $\mu$-superheavy and $Z$ is $\mu$-subheavy, then $Y \cap Z \not= \emptyset$.
\end{enumerate}
\end{thm}
\noindent
As an immediate corollary of Proposition~\ref{p:Basics}(iii) and Theorem~\ref{t:SHandsubheavy} we have:
\begin{coroll}\label{c:MainPoint}
	If $Y \subset V$ is $\mu$-subheavy and $Z \subset V$ is $\mu$-superheavy for
	a monotone quasi-morphism $\mu \fc \wt\Cont_{0}(V,\xi) \to \R$, then the following holds:
	\begin{enumerate}
	\item $Y$ cannot be displaced from $Z$, that is $\psi(Y) \cap Z \not= \emptyset$ for all $\psi \in \Cont_{0}(Y)$.
	\item $Z$ is nondisplaceable, that is $\psi(Z) \cap Z \not= \emptyset$ for all $\psi \in \Cont_{0}(Y)$. \qed
	\end{enumerate}
\end{coroll}
\noindent
See Section~\ref{s:RigidityProof} for the proofs of Proposition~\ref{p:Basics} and Theorem~\ref{t:SHandsubheavy},
which together with Corollary~\ref{c:MainPoint} are analogous to the basic properties of heavy and superheavy subsets of a symplectic manifold \cite[Section 1.4]{EntPol09R}.
We also have the following criterion for when a $\mu$-subheavy set is automatically $\mu$-superheavy, which we prove
in Section~\ref{s:RigidityProof}.
\begin{prop}\label{p:subheavyReebImpliesSuperheavy}
	Let $\mu\fc \wt\Cont_{0}(V, \xi) \to \R$ be a monotone quasi-morphism
	and $Y \subset V$ be a $\mu$-subheavy subset.  If $Y$ is preserved by the flow of some positive contact vector field,
	then $Y$ is $\mu$-superheavy.
\end{prop}
\noindent
In the context of Theorem~\ref{thm_ART}, note that Proposition~\ref{p:subheavyReebImpliesSuperheavy} implies
that the $\mu$-subheavy subset $Y \subset V$, which is strictly coisotropic, is actually $\mu$-superheavy.

As the next theorem shows, the properties of being subheavy and superheavy are respected by
the reduction of the quasi-morphisms in Theorem~\ref{thm_ART}.
Recall in Theorem~\ref{thm_ART} one has contact manifolds $(V, \xi, \alpha)$ and $(\ol{V}, \ol{\xi}, \ol{\alpha})$ 
and a closed submanifold $Y \subset V$ with a fiber bundle $\rho\fc Y \to \ol{V}$.
There are monotone quasi-morphisms
\begin{equation}\label{e:QuasimorphismsFromART}
	\mu\fc \wt\Cont_{0}(V, \xi) \to \R \quad\mbox{and}\quad \ol{\mu}\fc \wt{\Cont}_{0}(\ol{V}, \ol{\xi}) \to \R
\end{equation}
where by definition $\ol{\mu}(\wt\phi_{\ol{h}}) := \mu(\wt\phi_{h})$ for $h \in C^{\infty}([0,1] \times V)$ being any contact
Hamiltonian that satisfies $\rho^{*}\ol{h} = h|_{[0,1] \times Y}$.  
\begin{thm}\label{thm_rigidity_descends_ART}
	For monotone, $C^0$-continuous quasi-morphisms \eqref{e:QuasimorphismsFromART} from Theorem~\ref{thm_ART},
	if $Z \subset V$ is $\mu$-subheavy, then
	$\rho(Y \cap Z) \subset \overline{V}$ is $\overline{\mu}$-subheavy and likewise for
	superheavy sets.
\end{thm}
\begin{proof}
	First note that since $Y$ is $\mu$-superheavy by Proposition~\ref{p:subheavyReebImpliesSuperheavy}
	there is a non-trivial intersection $Y \cap Z \not= \emptyset$ by Theorem~\ref{t:SHandsubheavy}
	if $Z$ is $\mu$-subheavy.
	
	Assume $Z \subset V$ is $\mu$-superheavy and let $\ol{h} \in C^{\infty}(\ol{V})$ be such that 
	$\ol{h}|_{\rho(Y \cap Z)} > 0$.  For $\epsilon > 0$ sufficiently small, 
	let $\ol{f} \in C^{\infty}(\ol{V})$ be such that $\ol{f} = \epsilon$ in
	a neighborhood of $\rho(Y \cap Z)$ and $\ol{h} \geq \ol{f}$.
	Now we can pick an extension $f \in C^{\infty}(V)$ so that $f|_{Z} = \epsilon$ and $\rho^{*}\ol{f} = f|_{Y}$.
	Since $Z$ is $\mu$-superheavy it follows that $\ol{\mu}(\wt\phi_{\ol{f}}) = \mu(\wt\phi_{f}) > 0$,
	and hence by monotonicity $\ol{\mu}(\wt\phi_{\ol{h}}) > 0$.  Therefore
	$\rho(Y \cap Z) \subset \ol{V}$ is $\ol{\mu}$-superheavy. 
	
	Assume $Z \subset V$ is $\mu$-subheavy and let $\ol{h} \in C^{\infty}(\ol{V})$ be such that 
	$\ol{h}|_{\rho(Y \cap Z)} = 0$.  Pick a sequence $\ol{f}_{n} \in C^{\infty}(\ol{V})$
	such that there is $C^0$-convergence $\ol{f}_{n} \to \ol{h}$ and there are neighborhoods
	$\cN_{n}$ of $\rho(Y \cap Z)$ such that $\ol{f}_{n}|_{\cN_{n}} = 0$.
	We can pick extensions $f_{n} \in C^{\infty}(V)$ so that
	$f_{n}|_{Z} = 0$ and $\rho^{*}\ol{f}_{n} = f_{n}|_{Y}$.
	Since $Z$ is $\mu$-subheavy it follows that $\ol{\mu}(\wt\phi_{\ol{f}_{n}}) := \mu(\wt\phi_{f_{n}}) = 0$
	and hence $\ol{\mu}(\wt\phi_h) = 0$ since $\ol{\mu}$ is $C^0$-continuous.
	Therefore $\rho(Y \cap Z) \subset \ol{V}$ is $\ol{\mu}$-subheavy.
\end{proof}

\begin{rem}
We did not use the assumption of $C^0$-continuity of $\mu$ to prove that superheaviness descends under reduction.
Also the descent for subheaviness holds without the $C^0$-continuity assumption if $Z$ intersects $Y$
sufficiently nicely, for instance if there is a small tubular neighborhood $\pr \fc U \to Y$ of $Y$ such that $\pr|_{U \cap Z} \fc U \cap Z \to Y \cap Z$ is a fiber bundle.  However in general it is not possible to find a smooth extension $h$ of $\rho^*\ol h$ with $h|_Z = 0$, which we get around by using the $C^0$-continuity assumption.
\end{rem}

\begin{rem}\label{r:rigidSetsClosed}
We only consider closed subsets in the hierarchy of subheavy and superheavy subsets, and for instance we use this assumption in our proof of Theorem \ref{t:SHandsubheavy} and Proposition~\ref{p:subheavyReebImpliesSuperheavy}. Of course it is possible to extend the definitions and the theorems to arbitrary subsets via closure, but we have suppressed this for the sake of exposition.
\end{rem}
%%%%%%%%%%%%%%%%%%%%%%%

%%%%%%%%%%%%%%%%%%%%%%%
\subsubsection{Rigid Legendrians and pre-Lagrangians}

As demonstrated by previous work in contact rigidity \cite{Eli91N, EliHofSal95L,Ono96L,EliPol00P} two 
important classes of submanifolds in contact manifolds are
Legendrians and pre-Lagrangians.  Recall \cite[Section 2.2]{EliHofSal95L} that a 
\textbf{pre-Lagrangian} submanifold $Y^{n+1} \subset (V^{2n+1}, \xi)$ is one such that
$Y$ is transverse to $\xi$ and there is a contact form $\alpha$ such that $d\alpha|_{Y} = 0$,
that is $Y$ is a strictly coisotropic submanifold of minimal dimension.
An equivalent definition from \cite[Proposition 2.2.2]{EliHofSal95L} is that
$Y$ is the diffeomorphic image of a Lagrangian under the projection $SV \to V$ where $SV$ is
the symplectization of $V$. 
A nice class of examples is as follows: for a prequantization $\pi\fc (V, \alpha) \to (M, \w)$ and a Lagrangian 
$L \subset M$, the submanifold $\pi^{-1}(L) \subset V$ is pre-Lagrangian.
Note that Proposition~\ref{p:subheavyReebImpliesSuperheavy} implies every closed subheavy pre-Lagrangian submanifold is superheavy.  

As we will see from our examples of subheavy and superheavy subsets of contact manifolds
in Section~\ref{s:RigidityExamples}, prototypically a subheavy submanifold is a Legendrian and a superheavy submanifold is
a pre-Lagrangian.  In particular in Corollary~\ref{c:RealPartContact} we explicitly identify a $\mu$-subheavy Legendrian submanifold and a 
$\mu$-superheavy pre-Lagrangian torus for each of the quasi-morphisms in Theorem~\ref{thm_main}.  
For the case of Givental's quasi-morphism on $\RP^{2d-1}$, a $\mu_{\Giv}$-subheavy Legendrian is
$$
	\RP^{d-1}_L:= \{[z] \in \RP^{2d-1} \mid z \in \R^{d}\}
$$
and a $\mu_{\Giv}$-superheavy pre-Lagrangian torus is
$$
	T_{\RP}:= \{[z] \in \RP^{2d-1} \mid \abs{z_{1}}^{2} = \dots = \abs{z_{d}}^{2} = 1/\pi\},
$$
where we are viewing $\RP^{2d-1}$ as the quotient of the sphere $S^{2d-1} \subset \C^{d}$ with radius 
$\sqrt{d/\pi}$.  See Lemmas~\ref{exam_Legendrian_RP_subheavy} and \ref{exam_torus_RP_superheavy} for the proofs.

More generally we have the following existence theorem for nondisplaceable pre-Lag\-ran\-gi\-an tori, analogous
to Entov--Polterovich's proof \cite[Theorem 2.1]{EntPol06Q} of the existence of nondisplaceable Lagrangians in closed toric symplectic manifolds.
\begin{thm}\label{t:fiberSH}
	If $\mu\fc \wt\Cont_{0}(V, \xi) \to \R$ is a monotone quasi-morphism with the vanishing property
	and $(V, \alpha)$ is a prequantization of a closed toric manifold $(M, \w)$, then $V$ contains a nondisplaceable pre-Lagrangian torus.
\end{thm}
\noindent
See Section~\ref{s:RigidityProof} for the proof of Theorem~\ref{t:fiberSH}.
%%%%%%%%%%%%%%%%%%%%%%%

%%%%%%%%%%%%%%%%%%%%%%%%%%%%%
%%%%%%%%%%%%%%%%%%%%%%%%%%%%%

%%%%%%%%%%%%%%%%%%%%%%%%%%%%%
%%%%%%%%%%%%%%%%%%%%%%%%%%%%%
\subsection{Quasi-morphisms on $\wt\Ham(M)$ and symplectic quasi-states}\label{s:SymplecticQMQS}

For a closed symplectic manifold $(M, \w)$, a smooth Hamiltonian $F \fc [0,1] \times M \to \R$ induces a 
time-dependent vector field $\{X_{F_{t}}\}_{t \in [0,1]}$ by \begin{equation}\label{e:Ham}
\iota_{X_{F_{t}}} \w = -dF_{t} \quad\mbox{where } F_{t} = F(t, \cdot).
\end{equation}
Integrating $X_{F_{t}}$ gives a Hamiltonian isotopy $\{\phi_{F}^{t}\}_{t \in [0,1]}$ of $M$ based at $\id$
the identity of $M$
and these are in bijection with smooth Hamiltonians $F \fc [0,1] \times M \to \R$ \emph{normalized} so 
$\int_{M} F_{t} \,\w^{n} = 0$ for all $t \in [0,1]$.  The Hamiltonian group $\Ham(M)$ is the set of time-one maps $\phi_F^1$ of 
such Hamiltonian isotopies
and $\wt\Ham(M)$ is its universal cover. We write $\wt\phi_{H}$ for the element of $\wt\Ham(M)$ represented by the Hamiltonian isotopy
$\{\phi_{H}^{t}\}_{t \in [0,1]}$.
For normalized functions $H, G \in C^{\infty}(M)$ their Poisson bracket
\begin{equation}\label{e:HamPB}
	\{H, G\}_{\w}:= \w(X_{G}, X_{H}) = - dG(X_H)
\end{equation}
is the Hamiltonian whose vector field is the Lie bracket of $X_{H}$ and $X_{G}$.   A subset $S \subset M$
is \textbf{displaceable} if there is $\phi \in \Ham(M)$ so that $\phi(S) \cap \ol{S} = \emptyset$.

For a quasi-morphism $\mu_{M}\fc \wt\Ham(M) \to \R$ one defines the following two properties \cite{EntPol03C, EntPolZap07Q}:
\begin{enumerate}
	\item[(i)] \textbf{Stable:} For normalized Hamiltonians $H, G\fc [0,1] \times M \to \R$
	\begin{equation}\label{e:stable}
		\int_{0}^{1}\min_{M}(H_{t} - G_{t})\,dt \leq \frac{\mu_{M}(\wt{\phi}_{G}) - \mu_{M}(\wt{\phi}_{H})}{\Vol(M)} \leq
		\int_{0}^{1} \max_{M}(H_{t} - G_{t})\,dt\,,
	\end{equation}
where $\Vol(M) = \int_M \omega^n$.
	\item[(ii)] \textbf{Calabi:} If $U \subset M$ is an open displaceable subset
	and if $H\fc [0,1] \times M \to \R$ has support in $[0,1] \times U$, then
	$$
		\mu_{M}(\wt{\phi}_{H}) = \mbox{Cal}_{U}(\wt{\phi}_{H}) := \int_{0}^{1} \int_{U} H_{t}\,\w^{n} dt\,.
	$$
\end{enumerate}
Such quasi-morphisms were constructed by Entov--Polterovich in \cite{EntPol03C} using spectral invariants in Hamiltonian Floer theory 
and their construction has been refined and extended in \cite{EntPol08S, FOOO11S, Lan11Q, Lan13H, Lan13Q, MonVicZap12P,
Ost06C, Ush11D}.

On a closed symplectic manifold $(M, \w)$ a \textbf{quasi-state}
is a functional $\zeta\fc C^{\infty}(M) \to \R$ satisfying
the following properties for all $H, K \in C^{\infty}(M)$:
\begin{enumerate}
\item[(i)]\textbf{Monotone:} If $H \leq K$, then $\zeta(H) \leq \zeta(K)$.
\item[(ii)]\textbf{Normalized:} $\zeta(1) = 1$.
\item[(iii)] \textbf{Quasi-linearity:} If $\{H, K\}_{\w} = 0$, then $\zeta(H+K) = \zeta(H) + \zeta(K)$.
\end{enumerate}
These quasi-states are the symplectic version of Aarnes' notion of topological quasi-state \cite{Aarnes_qss_qms}.
As established in \cite{EntPol06Q}, every stable quasi-morphism $\mu_{M}\fc \wt{\Ham}(M) \to \R$ induces a 
quasi-state $\zeta_{\mu_{M}}\fc C^{\infty}(M) \to \R$ defined by
\begin{equation}\label{e:QuasiStateFormula}
	\zeta_{\mu_{M}}(H) := \frac{\int_{M}H\,\w^{n} - \mu_{M}(\wt{\phi}_{H})}{\Vol(M)}\,.
\end{equation}
Such quasi-states are $\Ham(M)$-invariant. If $\mu_M$ also has the Calabi property, then $\zeta_{\mu_M}$ has the \tb{vanishing property}, that is
$\zeta(H) = 0$ whenever $\supp(H) \subset M$ is displaceable.
\begin{defin}
	Let $\zeta\fc C^{\infty}(M) \to \R$ be a quasi-state on a closed symplectic manifold $(M, \w)$.  A closed subset $X \subset M$ 
	is \textbf{$\zeta$-superheavy} if
	\begin{equation}\label{e:QSSH}
		\min_{X} H \leq \zeta(H) \leq \max_{X} H
	\end{equation}
	for all $H \in C^{\infty}(M)$.
\end{defin}
\noindent
This definition was introduced in \cite{EntPol09R} and $\zeta$-superheavy sets
$X \subset M$ are nondisplaceable when $\zeta$ is $\Ham(M)$-invariant,
by \cite[Theorem 1.4]{EntPol09R}.
See \cite{BirEntPol04C, BuhEntPol12P, EntPol06Q,  EntPol09C, EntPol09R, EntPol10C, EntPolPy12O, EntPolZap07Q, FOOO11S, Kha09H} for various applications of Entov--Polterovich's quasi-morphisms and quasi-states. 

Recall for a prequantization $\pi\fc (V, \alpha) \to (M, \w)$ one has the following central extension of Lie algebras
$$0 \to \R \to (C^\infty(V)^{S^1},\{\cdot,\cdot\}_\alpha) \to (C^\infty(M)/\R,\{\cdot,\cdot\}_\omega) \to 0\,.$$
Here $C^\infty(V)^{S^1} \simeq C^\infty(M)$ is the set of $S^1$-invariant functions on $V$ and $C^\infty(M)/\R$ is canonically the Lie algebra of $\Ham(M)$.  When $M$ is closed this sequence has a unique splitting by the Lie algebra homomorphism
$$
	\sigma\fc C^\infty(M)/\R \to C^\infty(V)^{S^1} \quad\mbox{given by}\quad H \mapsto \pi^*H - \frac{\int_MH\,\omega^n}{\Vol(M)}
$$
and $\sigma$ induces a homomorphism
\begin{equation}\label{e:PullBackHamiltonians}
	\pi^{*}\fc \wt\Ham(M) \to \wt\Cont_{0}(V) \quad\mbox{where}\quad \pi^{*}(\wt{\phi}_{H}) = \wt{\phi}_{\sigma(H)}\,.
\end{equation}
See \cite[Section 1.3]{Ben10T} for more details on this point and in particular a proof that \eqref{e:PullBackHamiltonians} is
a homomorphism.

We now have the following result, generalizing Ben Simon \cite{Ben07T}, which uses the homomorphism \eqref{e:PullBackHamiltonians} to relate quasi-morphisms on $\wt{\Cont}_0$ and $\wt{\Ham}$.  Recall that $\wt\phi_{1} \in \wt\Cont_0(V)$ is the element generated by the Reeb vector field $R_{\alpha}$.
\begin{thm}\label{thm_HamQuasimorphism}
	Let $\pi\fc(V,\alpha) \to (M, \w)$ be a prequantization of a closed symplectic manifold and let
	$\mu\fc\wt\Cont_{0}(V) \to \R$ be a monotone quasi-morphism, then
	\begin{equation}\label{e:muM}
		\mu_{M} := -\frac{\Vol(M)}{\mu(\wt{\phi}_{1})} (\mu \circ \pi^{*})\fc \wt{\Ham}(M) \to \R
	\end{equation}
	is a stable quasi-morphism. The quasi-state associated to $\mu_{M}$ from \eqref{e:QuasiStateFormula} has
	the form:
	$$
		\zeta_{\mu_{M}}(H) := \frac{\mu(\wt{\phi}_{\pi^*\! H})}{\mu(\wt{\phi}_{1})}\,.
	$$
	If $\mu$ has the vanishing property, then $\mu_{M}$ has the Calabi property and $\zeta_{\mu_{M}}$ has the vanishing property.
\end{thm}

A historical remark is in order.  While Givental \cite{Giv90N} applied his quasi-morphism to various contact rigidity phenomena on $\RP^{2d-1}$, such as the existence of Reeb chords, it was first in the symplectic setting that Entov--Polterovich developed a systematic approach to use their quasi-morphisms in order to study symplectic rigidity. However as Theorem~\ref{thm_HamQuasimorphism} shows, for 
prequantizable symplectic manifolds, quasi-morphisms on $\wt{\Cont}_0$ are potentially more fundamental objects than quasi-morphisms on $\wt{\Ham}$.  A related question is if it is possible to obtain one of Entov--Polterovich's quasi-morphisms on $\wt{\Ham}(M)$ from a quasi-morphism on $\wt\Cont_{0}(V)$ via Theorem~\ref{thm_HamQuasimorphism}, and this is open even for the case of the prequantization $\RP^{3} \to \CP^1$.

The following proposition shows how the Entov--Polterovich notion of superheaviness \eqref{e:QSSH} with respect to a symplectic quasi-state on $(M, \w)$ is related to sub- and superheaviness with respect to a quasi-morphism on $\wt\Cont_{0}(V)$ when
$\pi\fc (V, \alpha) \to (M, \w)$ is a prequantization.

\begin{prop}\label{prop_prequant_rigidity}
If $\pi \fc (V,\alpha) \to (M,\omega)$ is a prequantization, $\mu \fc \wt\Cont_0(V) \to \R$ is a monotone quasi-morphism, and $\mu_M \fc \wt\Ham(M) \to \R$ is the quasi-morphism induced according to Theorem~\ref{thm_HamQuasimorphism}, then
\begin{enumerate}
\item if $Y \subset V$ is $\mu$-subheavy, then $\pi(Y) \subset M$ is $\zeta_{\mu_M}$-superheavy;
\item if $X\subset M$ is $\zeta_{\mu_M}$-superheavy, then $\pi^{-1}(X) \subset V$ is $\mu$-superheavy.
\end{enumerate}
\end{prop}

\noindent
Theorem~\ref{thm_HamQuasimorphism} and Proposition~\ref{prop_prequant_rigidity} are proved in Section~\ref{s:ProofPrequantRigid}.

Given a collection of $(H_1, \dots, H_k)$ pairwise Poisson commuting Hamiltonians on $M$, organized as a map $\Phi\fc M \to \R^{k}$, Entov--Polterovich in \cite{EntPol06Q} defined a fiber $\Phi^{-1}(p)$ to be a \textbf{stem} if
every other fiber $\Phi^{-1}(q) \subset M$ was displaceable.
They proved in \cite[Theorem 1.8]{EntPol09R} that a stem $X \subset M$ is superheavy
with respect to any quasi-state with the vanishing property.
Using Theorem~\ref{thm_HamQuasimorphism} and Proposition~\ref{prop_prequant_rigidity} we now have the following corollary
for any prequantization $\pi \fc (\wh{M}, \alpha) \to (M, \w)$
and monotone quasi-morphism $\mu\fc \wt{\Cont}_{0}(\wh{M}) \to \R$ with the vanishing property:
\begin{coroll}\label{c:Stem}
	If $X \subset (M, \w)$ is a stem, then $\pi^{-1}(X) \subset \wh{M}$ is $\mu$-superheavy. \qed
\end{coroll}
\noindent
Stems can be very singular subsets, an example being the product of $1$-skeletons
of fine triangulations of $2$-spheres \cite[Corollary 2.5]{EntPol06Q}.
%%%%%%%%%%%%%%%%%%%%%%%%%%%%%
%%%%%%%%%%%%%%%%%%%%%%%%%%%%%

%%%%%%%%%%%%%%%%%%%%%%%%%%%%%
%%%%%%%%%%%%%%%%%%%%%%%%%%%%%
\subsection{Examples of contact rigidity}\label{s:RigidityExamples}

In this subsection we will present concrete examples of subheavy and superheavy subsets of contact manifolds.

%%%%%%%%%%%%%%%%%%%%%%%
\subsubsection{Examples using Givental's quasi-morphism}

We will start with the rigidity results that just use Givental's monotone quasi-morphism
$\mu_{\Giv}\fc \wt\Cont_{0}(\RP^{2d-1}) \to \R$.  For us it will be convenient to introduce the following models of
the standard contact $S^{2d-1}$ and $\RP^{2d-1}$.
For $\gamma = (\gamma_1,\dots,\gamma_d)\in\N^d$, consider the sphere
\begin{equation}\label{e:SphereGamma}
	S^{2d-1}_\gamma = \{z\in\C^d\,|\,\textstyle\pi\sum_{j=1}^{d}\gamma_{j} |z_j|^2 = \sum_{j=1}^{d}\gamma_j\}
\end{equation}
with the contact form given by the restriction of
\begin{equation}\label{e:alphaSTD}
	\alpha_{\std} = \tfrac{1}{2} \sum_{j=1}^{d} (x_{j}\,dy_{j} - y_{j}\, dx_{j}) 
\end{equation}
to $S^{2d-1}_{\g}$ with Reeb flow
\begin{equation}\label{e:ReebGamma}
	\phi^{t}_{R_{\g}}(z_{1}, \dots, z_{d}) = (e^{2\pi i \g_{1} t/d}z_{1}, \dots, e^{2\pi i \g_{d} t/d}z_{d}).
\end{equation}
For the antipodal $\Z_{2}$-action on $\C^{d}$, let
\begin{equation}\label{e:RPgamma}
	(\R P^{2d-1}_\gamma, \xi_{\gamma}) := (S^{2d-1}_\gamma/\Z_2, \ker \alpha_{\std}).
\end{equation}
Note when $\g = (1, \dots, 1)$ that $(\RP^{2d-1}_{\g}, \xi_{\g})$ is the standard model for $(\RP^{2d-1}, \xi)$,
so we will drop the reference to $\g$ in this case.
Via radial projection $z \mapsto \tfrac{\sqrt{d}}{\sqrt{\pi}} \tfrac{z}{\abs{z}}$, which induces a contactomorphism
\begin{equation}\label{e:RescaleContact}
	r\fc(\R P^{2d-1}_{\gamma}, \xi_{\g}) \to (\R P^{2d-1}, \xi)\,,
\end{equation}
we have Givental's quasi-morphism $\mu_{\Giv}\fc \wt\Cont_{0}(\RP^{2d-1}_{\g}) \to \R$
for any $\g \in \N^{d}$.
\begin{lemma}\label{exam_torus_RP_superheavy}
	The torus
	\begin{equation}\label{e:Clifford}
		T_{\RP} := \{[z] \in \RP^{2d-1}_{\g} \mid \abs{z_{1}}^{2} = \cdots = \abs{z_{d}}^{2} = 1/\pi\} \subset \RP^{2d-1}_{\g}
	\end{equation}
	is $\mu_{\Giv}$-superheavy.	
\end{lemma}
\begin{proof}[Proof of Lemma~\ref{exam_torus_RP_superheavy}]
	Since the radial projection \eqref{e:RescaleContact} preserves $T_{\RP}$,
	it suffices to show $T_{\RP} \subset \RP^{2d-1}$ is $\mu_{\Giv}$-superheavy.
	Consider the prequantization
	$\pi\fc \RP^{2d-1} \to \CP^{d-1}$
	where we take
	$$\CP^{d-1} = \{[z_{1}: \cdots : z_{d}] \mid \pi \sum \abs{z_{j}}^{2} = d\}.$$
	Using the Hamiltonian $U(d)$-action on $\CP^{d-1}$, the 
	Clifford torus $\T_{\Clif}^{d-1} := \pi(T_{\RP})$ can be shown to be a stem
	\cite[Lemma 5.1]{BirEntPol04C}.
	Since $\mu_{\Giv}$ has the vanishing property by
	Proposition~\ref{prop_ANLMI_has_vanishing}, it follows from
	Corollary~\ref{c:Stem} that $T_{\RP}$ is $\mu_{\Giv}$-superheavy.
\end{proof}
\noindent
Lemma~\ref{exam_torus_RP_superheavy}
will play a large role in our proof of Theorem~\ref{thm_main}
for it will ensure we are applying Theorem~\ref{thm_ART} to a $\mu_{\Giv}$-superheavy subset.

While by Theorem~\ref{t:SHandsubheavy} it is impossible for a Legendrian submanifold to be superheavy, since they are always displaceable (for instance by an arbitrarily small positive contact isotopy), it is possible for a Legendrian to be subheavy as the next example shows.   The proof is given in Section~\ref{s:VanishingProof}.

\begin{lemma}\label{exam_Legendrian_RP_subheavy}
	The standard Legendrian
	\begin{equation}\label{e:StandardLegRP} 
		\R P^{d-1}_{L}:= \{[z] \in \RP^{2d-1}_{\g} \mid z \in \R^{d}\} \subset \RP^{2d-1}_{\g}
	\end{equation}
	is $\mu_{\Giv}$-subheavy.
\end{lemma}

\noindent
Once we take the orbit of $\RP^{d-1}_{L}$ under the Reeb flow, which is a closed subset since the Reeb flow is periodic, we get the following
immediate corollary of Lemma~\ref{exam_Legendrian_RP_subheavy} and 
Proposition~\ref{p:subheavyReebImpliesSuperheavy}.
\begin{coroll}\label{exam_Reeb_saturated_prelagrangian}
	The subset
	\begin{equation}\label{e:Lgamma}
		L_{\g} := \bigcup_{t \in \R} \phi_{R_{\gamma}}^{t}(\RP^{d-1}_{L}) \subset (\RP^{2d-1}_{\g}, \xi_{\g})
	\end{equation}
	is $\mu_{\Giv}$-superheavy. \qed
\end{coroll}

Corollary~\ref{exam_Reeb_saturated_prelagrangian} can be used to prove rigidity in weighted complex 
projective spaces. Recall for a primitive vector $\g \in \N^{d}$ that the weighted complex projective space $\C P(\gamma)$
is the symplectic orbifold obtained as the quotient of $S^{2d-1}_\gamma$ by the Reeb flow \eqref{e:ReebGamma}. 
A Hamiltonian isotopy of $\C P(\gamma)$ is by definition an isotopy that lifts to a contact isotopy of $S^{2d-1}_\gamma$ preserving the contact form.  If the fixed point set of the involution on $\CP(\g)$ induced by complex conjugation on $\C^d$ is
$$
	\RP(\g) \subset \CP(\g)
$$
and
$$
	T_{\CP} := \{[z] \in \CP(\g) \mid \abs{z_{1}}^{2} = \cdots = \abs{z_{d}}^{2} = 1/\pi\} \subset \CP(\g)
$$
then we have the following proposition.
\begin{prop}
	If for a primitive vector $\g = (\g_{1}, \dots, \g_{d}) \in \N^d$ each $\gamma_j$ is odd, then 
	$$
		\RP(\g) \cap \psi(\RP(\g)) \not= \emptyset\,, \quad \RP(\g) \cap \psi(T_{\CP})\not= \emptyset\,,
		\quad T_{\CP} \cap \psi(T_{\CP})\not= \emptyset
	$$
	for all Hamiltonian isotopies $\psi$ of $\C P(\gamma)$.
\end{prop}
\begin{proof}
	The fact that all $\gamma_j$ are odd is equivalent to the time $t = \tfrac{1}{2}$ Reeb flow \eqref{e:ReebGamma} 
	being the antipodal map on $S^{2d-1}_\gamma$.  Therefore if each $\g_{j}$ is odd, then 
	the quotient map $S^{2d-1}_\gamma \to \C P(\gamma)$ factors
	through the projection map
	$$\pi\fc \R P^{2d-1}_\gamma \to \C P(\gamma)$$
	and hence any Hamiltonian isotopy of $\CP(\g)$ lifts to a contact isotopy of $\RP^{2d-1}_{\g}$.

	By the definitions, under the projection map
	 $\pi(L_{\g}) \subset \RP(\g)$ and $\pi(T_{\RP}) = T_{\CP}$.
	Since $L_{\g}$ and $T_{\RP}$ are $\mu_{\Giv}$-superheavy by
	Lemmas~\ref{exam_Reeb_saturated_prelagrangian} and \ref{exam_torus_RP_superheavy}
	it follows from Theorem~\ref{t:SHandsubheavy}
	that both $L_{\g}$ and $T_{\RP}$ are nondisplaceable 
	and cannot be displaced from each other by a contact isotopy.
	Therefore the same holds for $\RP(\g)$ and $T_{\CP}$ for Hamiltonian isotopies.
\end{proof}
\noindent
Nondisplaceability of $T_{\CP} \subset \CP(\g)$ for any primitive $\g \in \N^{d}$ was proved by
Woodward \cite{Woo11G} and Cho--Poddar \cite{ChoPod12H}. Nondisplaceability of $\R P(\gamma)$ for an odd primitive vector $\gamma$ was previously proved by
Lu \cite{Lu08S}.

%%%%%%%%%%%%%%%%%%%%%%%

%%%%%%%%%%%%%%%%%%%%%%%
\subsubsection{Examples using the quasi-morphisms from Theorem~\ref{thm_main}}

In the proof of Theorem~\ref{thm_main} in Section~\ref{s:ProofMainThm} we
will apply Theorem~\ref{thm_ART} to Givental's quasi-morphism to build $\mu$.
In particular for an appropriate primitive vector $\g \in \N^{d}$ in Section~\ref{s:ApplyingART} we present
the prequantization $(\wh{M}, \alpha)$ in the setting \eqref{e:SetUpART} of Theorem~\ref{thm_ART}
\begin{equation}\label{e:SetUpApply}
	(\RP^{2d-1}_{\g}, \xi_{\g}, \alpha_{\std}) \supset (Y, \alpha_{\std}|_{Y}) \stackrel{\rho}{\longrightarrow} (\wh{M}, \xi, \a)
\end{equation}
where $Y$ is a $\mu_{\Giv}$-superheavy submanifold containing $T_{\RP}$ and $\mu := \ol{\mu}_{\Giv}$ is the reduction
of Givental's quasi-morphism.
For the torus $T_{\RP}$ and standard Legendrian $\RP_{L}^{d-1}$ in $\RP^{2d-1}_{\g}$ from
\eqref{e:Clifford} and \eqref{e:StandardLegRP} define the following two subsets of $\wh{M}$
\begin{equation}\label{e:ContactReal}
	T_{\wh{M}}:= \rho(T_{\RP}) \quad\mbox{and}\quad
	\wh M_\R := \rho(Y\cap \R P^{d-1}_L).
\end{equation}
Note that $T_{\wh M}$ is a pre-Lagrangian torus while $\wh M_\R$ is Legendrian. We now have the following corollary.
\begin{coroll}\label{c:RealPartContact}
	For a quasi-morphism $\mu\fc \wt{\Cont}_0(\wh{M}) \to \R$ from Theorem~\ref{thm_main},
	the pre-Lagrangian $T_{\wh{M}}$ is $\mu$-superheavy and the Legendrian $\wh{M}_{\R}$ is $\mu$-subheavy.
\end{coroll}
\begin{proof}
	This follows from Theorem~\ref{thm_rigidity_descends_ART}
	together with Lemma~\ref{exam_torus_RP_superheavy} and 
	Lemma~\ref{exam_Legendrian_RP_subheavy}.
\end{proof}

The next result concerns rigidity for the
real part $M_{\R} \subset (M, \w)$ of a symplectic toric manifold, which is characterized as
the fixed point set of the anti-symplectic involution that preserves the moment map.
Using the prequantization $\pi\fc (\wh{M}, \a) \to (M, \w)$
we construct in Section~\ref{s:BuildPrequantization}
for a monotone even toric manifold, the real part of $M$ can be identified with
$$
	M_{\R} := \pi(\wh M_\R)
$$
where $\wh{M}_{\R} \subset \wh{M}$ is from \eqref{e:ContactReal}.
For the quasi-morphism $\mu\fc \wt\Cont_{0}(\wh{M}) \to \R$ from Theorem~\ref{thm_main}, let
$\zeta_{\mu_{M}} \fc C^{\infty}(M) \to \R$ be the induced symplectic quasi-state on $(M, \w)$
from Theorem~\ref{thm_HamQuasimorphism}.

\begin{prop}\label{p:RealPart}
	The real part $M_{\R} \subset (M, \w)$ of a monotone even toric symplectic manifold
	is $\zeta_{\mu_{M}}$-superheavy and hence nondisplaceable.
\end{prop}
\begin{proof}
	Using $\wh{M}_{\R}$ is $\mu$-subheavy by Corollary~\ref{c:RealPartContact} 
	it follows $M_\R = \pi(\wh{M}_{\R})$ is $\zeta_{\mu_M}$-superheavy by
	Proposition~\ref{prop_prequant_rigidity} and therefore is nondisplaceable.
\end{proof}

\noindent
Haug~\cite{Hau13O} proved the nondisplaceability part of Proposition~\ref{p:RealPart} without the even assumption
using Biran--Cornea's Lagrangian quantum homology \cite{BirCor09A, BirCor09R}.

Similarly the central toric fiber $T_{M} \subset (M, \w)$
of a monotone even toric manifold is nondisplaceable and cannot be displaced from the
real part $M_{\R}$.  This is because
$\pi^{-1}(T_{M}) = T_{\wh{M}}$
so Proposition~\ref{prop_prequant_rigidity} and Corollary~\ref{c:RealPartContact}
imply $T_{M}$ is $\zeta_{\mu_{M}}$-superheavy.  The nondisplaceability now follows from \cite[Theorem 1.4]{EntPol09R}.
These results have been established by various authors \cite{AbrMac13R, AlsAmo12F, Cho08N, EntPol06Q, FOOO11L}.
In particular Abreu--Macarini \cite{AbrMac13R} showed how simple previous nondisplaceability results
in $\C P^n$ can be combined with symplectic reduction to prove the nondisplaceability results for $T_{M}$ and the pair
$(T_{M}, M_{\R})$, but could not prove $M_{\R}$ was nondisplaceable.  
%%%%%%%%%%%%%%%%%%%%%%%

%%%%%%%%%%%%%%%%%%%%%%%%%%%%%
%%%%%%%%%%%%%%%%%%%%%%%%%%%%%

%%%%%%%%%%%%%%%%%%%%%%%%%%%%%
%%%%%%%%%%%%%%%%%%%%%%%%%%%%%
\subsection{Orderability and metrics on $\wt{\Cont}_{0}$}

\noindent
Recall from Section~\ref{section_qms} 
that a contact manifold $(V, \xi)$ is orderable if $\wt{\Cont}_0(V,\xi)$ is partially ordered by the relation $\preceq$ from \eqref{e:PartialOrder}.

%%%%%%%%%%%%%%%%%%%%%%%%%%%%%
\subsubsection{Orderability for contact manifolds and quasi-morphisms}\label{section_orderability} 

There has been a fair amount of research concerning orderability of contact manifolds. Since we are mainly dealing with closed contact manifolds, let us give examples of orderable and non-orderable closed contact manifolds. 
Eliashberg--Kim--Polterovich prove in \cite{EliKimPol06G} that the ideal contact boundary of a sufficiently subcritical Weinstein manifold is not orderable. In particular the standard contact spheres $S^{2d-1}$ are not orderable for $d \geq 2$. 
Cosphere bundles of closed manifolds are known to be orderable \cite{AlbFra13E,CheNem10N,EliKimPol06G,EliPol00P}
and more generally Albers--Merry proved in \cite{AlbMer13O} that Liouville-fillable contact manifolds with nonvanishing 
Rabinowitz Floer homology are orderable.  Using the connection between orderability and contact squeezing developed by Eliashberg--Kim--Polterovich \cite{EliKimPol06G}, Milin \cite{Mil08O} and Sandon \cite{San11E} proved that lens spaces are orderable.

In \cite[Section 1.3.E]{EliPol00P} Eliashberg--Polterovich proved
that $\RP^{2d-1}$ is orderable using Givental's quasimorphism $\mu_{\Giv}$. Their argument works in general and implies the following.
\begin{thm}[{\cite{EliPol00P}}]\label{t:orderability}
	A contact manifold $(V, \xi)$ is orderable if there is 
	a monotone quasi-morphism $\mu$ on $\wt\Cont_0(V,\xi)$.
\end{thm}

\begin{proof}
	By \cite[Criterion 1.2.C]{EliPol00P} to prove $(V, \xi)$
	is orderable it suffices to prove $\id \not= \wt\phi_{h}$ in $\wt{\Cont}_{0}(V, \xi)$ 
	for any contact Hamiltonian with $h > 0$ on $[0,1] \times V$.  Since $\mu(\id) = 0$, we are done
	because for any such contact Hamiltonian $\mu(\wt{\phi}_{h}) > 0$ by Proposition~\ref{p:Basics}(iv).
\end{proof}

\begin{coroll}\label{coroll_orderable}
The contact manifolds $(\wh{M}, \xi)$ in Theorem~\ref{thm_main} are orderable. \qed
\end{coroll}

\noindent
Recall that the contact manifolds $(\wh{M}, \xi)$ are obtained from contact reduction of $\RP^{2d-1}$, which is of course orderable.  
It would be interesting to prove Corollary~\ref{coroll_orderable} directly, that is to prove orderability persists under contact reduction.  

By Theorem~\ref{t:orderability}, orderability is a necessary condition for the existence of a non-zero homogeneous monotone quasi-morphism on $\wt{\Cont}_{0}(V)$.  However in general the converse is not well understood and potentially is a delicate question, which
we will illustrate with the following examples regarding $\R^{2n} \times S^1$ and its group of compactly supported contactomorphisms $\Cont_{0}^{c}(\R^{2n} \times S^1)$, where the contact form is $\alpha_{\std} + dt$ and $dt$ is the angular form on $S^1 = \R/\Z$.

\begin{exam}
	Sandon has proved \cite{San10A, San11C} that $\R^{2n} \times S^1$ is orderable, that is the Eliashberg--Polterovich relation \eqref{e:PartialOrder} is indeed a partial order on the group $\wt\Cont{}_0^c(\R^{2n} \times S^1)$, and also proved it induces a
	% bi-invariant 
	partial order on $\Cont_{0}^{c}(\R^{2n} \times S^{1})$.
	However $\Cont_{0}^{c}(\R^{2n} \times S^1)$ admits no non-zero homogeneous quasi-morphisms,
	due to a general argument of Kotschick \cite[Theorem 4.2]{Kot08S}
	and the fact that $\Cont_{0}^{c}$ is always perfect due to Rybicki \cite{Rybicki_comms_contactomorphisms}.
\end{exam}

\begin{exam}\label{e:UCPerfect}
	This example rests on the speculation that $\wt{\Cont}{}_{0}^{c}$ is in general perfect.
	In this case again Kotschick's argument proves $\wt{\Cont}{}_{0}^{c}(\R^{2n} \times S^1)$ admits no non-zero 
	homogeneous quasi-morphisms, despite $\preceq$ being a partial order on  $\wt{\Cont}{}_{0}^{c}(\R^{2n} \times S^1)$.
\end{exam}

\begin{exam}
	Consider now the domain $B^{2n}_{R} \times S^1$ where 
	$$B^{2n}_{R} := \{z \in \C^{n} \mid \pi \abs{z}^2 < R\}.$$
	Since $\R^{2n} \times S^1$ is contactomorphic to $B^{2n}_{1} \times S^1$ by \cite[Proposition 1.24]{EliKimPol06G},
	Example~\ref{e:UCPerfect} indicates $\wt{\Cont}{}_{0}^{c}(B^{2n}_1 \times S^1)$ 
	does not admit a non-zero homogeneous quasi-morphism.
	
	On the other hand, $\wt{\Cont}{}_{0}^{c}(B^{2n}_R \times S^1)$ admits a non-zero homogeneous quasi-morphism
	whenever $\tfrac{2n}{n+1} < R < 2$.  When $R < 2$ we have the contact embedding
	\begin{equation}\label{embedding}
		\Phi\fc B^{2n}_R \times S^1 \to \RP^{2n+1} \quad\mbox{by}\quad (z,t) \mapsto 
		e^{\pi i t} \sqrt{\tfrac{n+1}{2}} \Big(z, \sqrt{\tfrac{2}{\pi} - \abs{z}^2}\Big)
	\end{equation}
	written as a map to $S^{2n+1} = \{z \in \C^{n+1} \mid \pi \abs{z}^2 = n+1\}$ from \eqref{e:SphereGamma}; where $t\in[0,1)$.
	When $R > \tfrac{2n}{n+1}$, one can check the image $\Phi(B^{2n}_R)$ contains
	the $\mu_{\Giv}$-superheavy torus $T_{\RP} \subset \RP^{2n+1}$ from Lemma~\ref{exam_torus_RP_superheavy}
	and hence one can use $\Phi$ to pull-back $\mu_{\Giv}$ to a non-zero homogeneous quasi-morphism on 
	$\wt{\Cont}{}_{0}^{c}(B^{2n}_{R} \times S^1)$.
\end{exam}

The reader is also referred to Ben Simon and Hartnick's work \cite{BenSimon_Hartnick_Quasi_total_orders_translation_numbers, BenHar12R} regarding a general connection between quasi-morphisms and partial orders.
%%%%%%%%%%%%%%%%%%%%%%%%%%%%%

%%%%%%%%%%%%%%%%%%%%%%%%%%%%%
\subsubsection{Sandon-type metric}

In \cite{San10A} Sandon introduced an unbounded integer-valued conjugation-invariant norm on $\Cont_{0}^{c}(\R^{2n}\times S^{1})$,
the identity component of the group of compactly supported contactomorphisms of $\R^{2n} \times S^{1}$, and such norms have
been further studied in \cite{AlbMer13O, ColSan12T, FraPolRos12O, Zap13G}.  In what follows we will consider the norm $\nu$
defined in \cite{FraPolRos12O}, whose definition we will now recall.

Consider any orderable contact manifold $(V, \xi)$ for which there is a positive contact Hamiltonian $f > 0$ such that
$\wt\phi_{f}$ is in the center of $\wt\Cont_{0}(V)$.  Examples of this are given by orderable contact manifolds with a periodic Reeb flow,
for instance $\RP^{2d-1}$ or any of the contact manifolds $\wh{M}$ from Theorem~\ref{thm_main}.
The functionals on $\wt\Cont_{0}(V)$
$$
	\nu_{-}(\wt\psi) := \max\big\{n \in \Z \,|\, \wt\phi_{f}^{n} \preceq \wt\psi\big\}
	\quad\mbox{and}\quad	
	\nu_{+}(\wt\psi) := \min\big\{n \in \Z \,|\, \wt\psi \preceq \wt\phi_{f}^{n}\big\}
$$
are conjugation-invariant, since $\wt\phi_{f}$ is in the center of $\wt\Cont_{0}(V)$, 
and
$$
	\nu\fc \wt{\Cont}_{0}(V) \to \Z \quad\mbox{where}\quad
	\nu(\wt\phi):= \max\big\{ \big|\nu_{+}(\wt\phi)\big|, \big|\nu_{-}(\wt\phi)\big| \big\},
$$
defines a conjugation-invariant norm, by \cite[Theorem 2.4]{FraPolRos12O}.
Using $\wt\phi_{f}$ is generated by a strictly positive
contact Hamiltonian, it is easy to see from \cite[Criterion 1.2.C]{EliPol00P} that
$\nu(\wt\phi_{f}^{n}) = \abs{n}$ for any $n \in \Z$ and hence $\nu$ is stably unbounded.
This norm is related to monotone quasi-morphisms on $\wt{\Cont}_{0}(V)$ as follows:
\begin{lemma}
	If $\mu\fc \wt{\Cont}_{0}(V) \to \R$ is a monotone quasi-morphism, then
	$$
		\big|\mu(\wt\psi)\big| \leq \mu(\wt{\phi}_{f})\, \nu(\wt{\psi})\,.
	$$
\end{lemma}
\noindent Note that $\mu(\wt{\phi}_{f}) > 0$ since $\mu \neq 0$.
\begin{proof}
	By the definition of $\nu_{\pm}$ and the fact that $\mu$ is monotone and homogeneous we have
	$$
		\nu_{-}(\wt{\psi})\, \mu(\wt{\phi}_{f}) \leq \mu(\wt{\psi}) \leq \nu_{+}(\wt{\psi})\, \mu(\wt{\phi}_{f})
	$$
	from which the result follows.
\end{proof}

Next we show that the above norm is unbounded on subgroups of $\wt\Cont_0(V)$ associated to certain open subsets. For an open subset $U \subset V$ we let $\wt\Cont_0(U) \subset \wt\Cont_0(V)$ be the subgroup consisting of elements $\wt\phi_h$ where
the Hamiltonian $h$ has compact support contained in $U$.
\begin{thm}
If $U\subset V$ is an open subset containing a $\mu$-superheavy subset, then there is $\wt\psi \in \wt\Cont_0(U)$ with 
$$\lim_{n \to \infty} \frac{\nu(\wt\psi^n)}{n} > 0$$
that is $\nu$ is stably unbounded on $\wt\Cont_0(U)$.
\end{thm}
\begin{proof}
By the above lemma we have
$$\nu(\wt\psi) \geq \frac{|\mu(\wt\psi)|}{\mu(\wt\phi_f)}\,,$$
therefore it suffices to produce an element $\wt\psi \in \wt\Cont_0(U)$ with $\mu(\wt\psi)\neq 0$. If $h$ is such that the restriction of $h$ to the superheavy subset is positive and $\supp (h) \subset U$, then since by definition $\mu(\wt\phi_h) > 0$, we are done.
\end{proof}

Colin--Sandon in \cite{ColSan12T} used the notion of a discriminant point
to define a non-degenerate bi-invariant metric on $\wt\Cont_{0}(V, \xi)$ for any contact manifold, which they called the discriminant metric.  
Using the relation between Givental's quasi-morphism $\mu_{\Giv}$ with discriminant points, 
see Section~\ref{s:discriminant} for more on this, Colin--Sandon were able to show the discriminant metric is stably unbounded on
$\wt{\Cont}_0(\RP^{2d-1})$.  It would be interesting to determine if the quasi-morphism $\mu\fc \wt{\Cont}_0(\wh{M}) \to \R$
we built in Theorem~\ref{thm_main} can also be used to show the discriminant metric on $\wt{\Cont}_0(\wh{M})$ is stably unbounded.
%%%%%%%%%%%%%%%%%%%%%%%%%%%%%

%%%%%%%%%%%%%%%%%%%%%%%%%%%%%
%%%%%%%%%%%%%%%%%%%%%%%%%%%%%

%%%%%%%%%%%%%%%%%%%%%%%%%%%%%
%%%%%%%%%%%%%%%%%%%%%%%%%%%%%
\subsection{Examples of even monotone polytopes}\label{s:EvenExample}

Moment polytopes corresponding to closed monotone symplectic toric manifolds are known as smooth Fano polytopes.  They
have been classified by hand up to dimension $4$ in \cite{Bat81T,Bat99O,Sat00T,WatWat82T} and
there is an algorithm in \cite{Obr07A} for higher dimensions.  We will now give various examples of even smooth Fano polytopes
in $\R^{n}$ and their corresponding symplectic toric manifolds.  For the polytopes we
will just list the interior conormals
$\{\nu_{j}\} \in \Z^{n}$ where $\{\epsilon_{1}, \dots, \epsilon_{n}\}$ is the standard basis.

The first example is $\CP^{n}$ with conormals 
$\{\epsilon_{1}, \dots, \epsilon_{n}, -(\epsilon_{1} + \cdots + \epsilon_{n})\}$
and in dimension two there are 
$$
\mbox{$\CP^{2}$\,,\,\, $\CP^{1} \times \CP^{1}$\,,\,\, $\CP^{2}\#3\ol{\CP}^{2}$}
$$
where the last one has conormals $\{\pm\epsilon_{1}, \pm\epsilon_{2}, \pm(\epsilon_{1}+\epsilon_{2})\}$.
In dimension three there are $18$ smooth Fano polytopes by the classification \cite{Bat81T,WatWat82T}
and $8$ are even.  Four are basic 
$$
\mbox{$\CP^{3}$\,,\,\, $\CP^{1} \times \CP^{2}$\,,\,\, $(\CP^{1})^{3}$\,,\,\, $\CP^{1} \times (\CP^{2}\#3\ol{\CP}^{2}$)}
$$
and the remaining four have the structure of toric bundles \cite[Definition 3.10]{McDTol10P}:
\begin{figure}[h]
   \centering
   \def\svgwidth{350pt}
   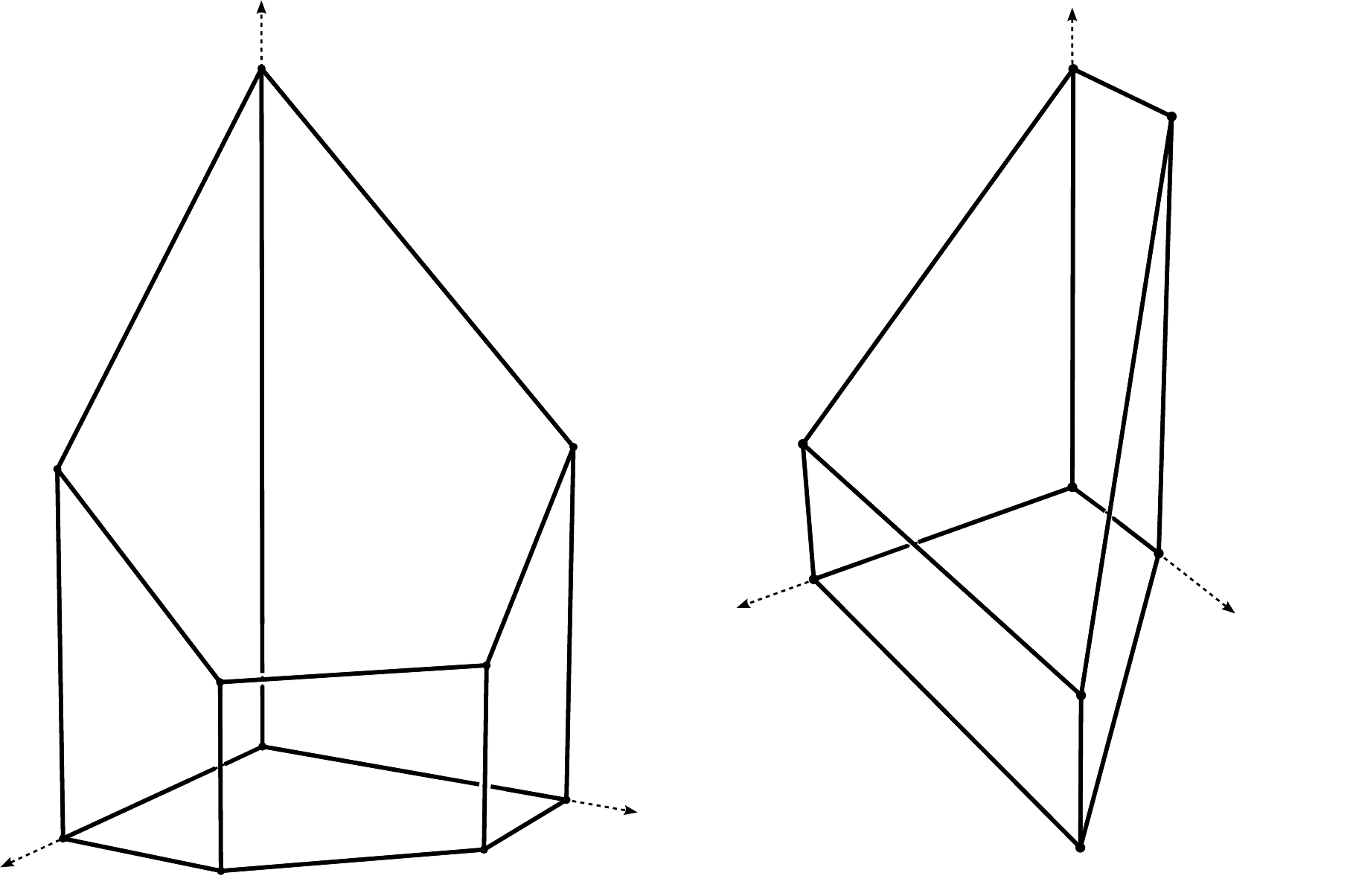
   \caption{Left: The polytope for the $(\CP^{2}\#2\ol{\CP}^{2})$-bundle over $\CP^{1}$ in (ii).  Right: The
    polytope for the $\CP^{1}$-bundle $\bP(\C \oplus \cO(1,-1))$ over $\CP^{1} \times \CP^{1}$ in (iv).}
  \label{f:poly}
\end{figure} 
\begin{enumerate}
	\item the $\CP^{1}$-bundle $\bP(\C\oplus\cO(2))$ over $\CP^{2}$ with conormals
	$$
		\{\pm \epsilon_{1}\,,\,\, \epsilon_{2}\,,\,\, \epsilon_{3}\,,\,\, 2\epsilon_{1} - \epsilon_{2}-\epsilon_{3}\}
	$$
	\item the $(\CP^{2}\#2\ol{\CP}^{2})$-bundle $F_{3}^{4}$ (in the notation of \cite{WatWat82T}) over $\CP^{1}$ with conormals
	$$
		\{\pm \epsilon_{1}\,,\,\, \pm\epsilon_{2}\,,\,\, -\epsilon_{1}-\epsilon_{2}\,,\,\,\epsilon_{3}\,,\,\,
		-\epsilon_{1}-\epsilon_{2} - \epsilon_{3}\}
	$$
	\item the $\CP^{1}$-bundle $\bP(\C\oplus \cO(1,1))$ over $\CP^{1} \times \CP^{1}$ with conormals
	$$
		\{\pm\epsilon_{1}\,,\,\,\epsilon_{2}\,,\,\,\epsilon_{3}\,,\,\, \epsilon_{1} -\epsilon_{2}\,,\,\,\epsilon_{1}-\epsilon_{3}\}
	$$
	\item and the $\CP^{1}$-bundle $\bP(\C \oplus \cO(1,-1))$ over $\CP^{1} \times \CP^{1}$ with conormals
	$$
		\{ \pm\epsilon_{1}\,,\,\, \epsilon_{2}\,,\,\, \epsilon_{3}\,,\,\,\epsilon_{1} - \epsilon_{2}\,,\,\,-\epsilon_{1} -\epsilon_{3}\}.
	$$
\end{enumerate}
The example in (i) generalizes to the $\CP^{1}$-bundles $\bP(\C\oplus\cO(2k))$ over $\CP^{n}$
where $0 \leq 2k \leq n$.  See Figure~\ref{f:poly}
for the polytopes from (ii) and (iv).
%%%%%%%%%%%%%%%%%%%%%%%%%%%%%
%%%%%%%%%%%%%%%%%%%%%%%%%%%%%

%%%%%%%%%%%%%%%%%%%%%%%%%%%%%
%%%%%%%%%%%%%%%%%%%%%%%%%%%%%
\subsection*{Acknowledgements}

We are very grateful to Leonid Polterovich for his wonderful guidance and support with this paper. We thank Miguel Abreu, Alexander Givental, and Yael Karshon for helpful conversations and Milena Pabiniak, Sheila Sandon, and Egor Shelukhin for their comments on the manuscript. We also wish to thank the organizers and participants of the workshop ``From Conservative Dynamics to Symplectic and Contact Topology''
at the Lorentz Center (Leiden), the Geometry Seminar at Universit\'e Lyon 1 (Lyon), and the Geometry and Topology seminars at the University of Haifa and The Technion (Haifa), for the opportunity to present preliminary results of this work. Work on this paper began while visiting the Lorentz Center (Leiden) and our meetings at Tel Aviv University were instrumental in the completion of this project. 
The first named author thanks Leonid Polterovich and Yaron Ostrover for their warm hospitality during his stay at Tel Aviv University.
Part of the writing was completed while the second named author was a postdoc at LMU (Munich), and he would like to acknowledge the nourishing research atmosphere and hospitality of this university.  Finally we would like to thank the referee for their helpful comments,  suggestions, and corrections.

%%%%%%%%%%%%%%%%%%%%%%%%%%%%%
%%%%%%%%%%%%%%%%%%%%%%%%%%%%%

%%%%%%%%%%%%%%%%%%%%%%%%%%%%%%%%%%%%%%%%%%%%%%%%%%%%%%%%%%
%%%%%%%%%%%%%%%%%%%%%%%%%%%%%%%%%%%%%%%%%%%%%%%%%%%%%%%%%%

%%%%%%%%%%%%%%%%%%%%%%%%%%%%%%%%%%%%%%%%%%%%%%%%%%%%%%%%%%
%%%%%%%%%%%%%%%%%%%%%%%%%%%%%%%%%%%%%%%%%%%%%%%%%%%%%%%%%%
\section{Proof of the main theorem}\label{s:ProofMainThm}

In this section we will present the proof of Theorem \ref{thm_main}. 
Section \ref{s:BuildPrequantization} contains the construction of the prequantization $\pi\fc(\wh M, \alpha) \to (M, \w)$ and 
Section \ref{s:ApplyingART} builds the announced quasi-morphisms $\mu\fc\wt\Cont_0(\wh M) \to \R$ using Theorem~\ref{thm_ART}.

%%%%%%%%%%%%%%%%%%%%%%%%%%%%%
%%%%%%%%%%%%%%%%%%%%%%%%%%%%%
\subsection{Constructing the family of contact manifolds}\label{s:BuildPrequantization}

The goal of this subsection is to present the construction of a prequantization $(\wh M,\xi,\alpha)$ for an even closed monotone toric symplectic manifold $(M,\omega)$ with moment polytope
\begin{equation}\label{e:Polytope1}
	\Delta = \{x \in \mf{t}^{*} \,|\, \ip{x, \nu_{j}} + 1 \geq 0 \mbox{ for $j=1, \dots, d$}\}
\end{equation}
as in \eqref{e:Polytope} where $\nu_{j} \in \mf{t}_{\Z}$ are primitive vectors and each one defines a different facet of the polytope $\Delta$.
The polytope $\Delta$ is compact and smooth, meaning each $k$-codimensional face of $\Delta$ is the intersection of exactly
$k$ facets and the $k$ associated conormals $\{\nu_{l_1}, \dots, \nu_{l_{k}}\}$ can be extended to an integer basis for the lattice 
$\mf{t}_{\Z}$.
In \eqref{e:Polytope1} we have used the normalization $[\w] = c_{1}(M)$ since $(M, \w)$ is monotone and
scaling the polytope $\Delta$ is equivalent to scaling $\w$.

%%%%%%%%%%%%%%%%%%%%%%%%%%%%%
\subsubsection{The standard toric structure on $\C^{d}$ and Delzant's construction}

Let us briefly recall the standard toric structure on $(\C^{d}, \w_{\std} = dx \wedge dy)$.
The action of $\T^{d} = \R^{d}/\Z^{d}$ on $\C^{d}$, which rotates each coordinate, is induced by the moment map
$$
	P\fc \C^{d} \to \R^{d*} \quad\mbox{where}\quad \ip{\lambda, P}(z) = \pi \sum_{j=1}^{d} \lambda_{j}\abs{z_{j}}^{2}
	\quad\mbox{for $\lambda = (\lambda_{1}, \ldots, \lambda_{d}) \in \R^{d}$}.
$$
Indeed for $\lambda \in \R^{d}$, the vector field 
\begin{equation}\label{e:Xvector}
	X_{\lambda}(z) = 2\pi i (\lambda_{1}z_{1}, \dots, \lambda_{d}z_{d}) \in \C^{d} = T_{z}\C^{d}
\end{equation}
is the Hamiltonian vector field for the function $\ip{\lambda, P}\fc \C^{d} \to \R$ and it
gives the infinitesimal action of $\lambda$ on $\C^{d}$.
Observe for the $1$-form
\begin{equation}\label{e:Alpha}
\a_{\std} = \tfrac{1}{2} \sum_{j=1}^{d} (x_{j}\,dy_{j} - y_{j}\,dx_{j})
\end{equation}
where $d\a_{\std} = \w_{\std}$ one has
\begin{equation}\label{e:AlphaStuff}
	\a_{\std}(X_{\lambda}) = \ip{\lambda, P}
	\quad\mbox{and}\quad
	\iota_{X_{\lambda}}d\a_{\std} = \iota_{X_{\lambda}}\w_{\std}= -d\ip{\lambda, P}.
\end{equation}

%%%%%%%%%%%%%%%%%%%%%%%%%%%%%

Delzant in \cite{Del88H} gave a way to reconstruct a closed symplectic toric manifold from its moment polytope using symplectic reduction of $\C^{d}$, which we will now recall in the case of the polytope $\Delta$ in \eqref{e:Polytope1}.
Define the surjective linear map 
$$
	\beta_{\Delta}\fc\R^{d} \to \mf{t} \quad\mbox{by}\quad \epsilon_{j} \mapsto \nu_{j} \quad\mbox{for $j=1, \dots, d$}
$$ 
where $\{\epsilon_{j}\}_{j=1}^{d}$ are the standard basis vectors of $\R^{d}$ and $\nu_{j} \in \mf{t}_{\Z}$ are 
conormals in \eqref{e:Polytope1}.  Since $\Delta$ is compact and smooth, we know $\beta_{\Delta}(\Z^d) = \mf{t}_\Z$,
and so we can define the connected subtorus 
\begin{equation}\label{e:defK}
	\bK \leq \T^{d} \quad\mbox{to be the kernel of the induced map $[\beta_{\Delta}]\fc \T^d \to \T$}
\end{equation}
with Lie algebra 
\begin{equation}\label{e:Kernel}
	\mf{k}:= \ker(\beta_{\Delta}\fc \R^d \to \mf{t}).
\end{equation}
If $\iota^{*}\fc \R^{d*} \to \mf{k}^{*}$ is dual to the inclusion $\mf{k} \subset \R^{d}$, then
the action of $\bK$ on $\C^{d}$ has
$$
	P_{\bK} := \iota^{*} \circ P \fc \C^{d} \to \mf{k}^{*}
$$
for its moment map. The torus $\bK$ acts freely on the regular level set
\begin{equation}\label{e:LevelSet}
	P_{\bK}^{-1}(c) \subset \C^{d} \quad \mbox{where} \quad c:= \iota^{*}(1, \dots, 1) \in \mf{k}^{*}
\end{equation}
and for $\lambda \in \mf{k}$ it follows from \eqref{e:AlphaStuff} that $(\cL_{X_{\lambda}}\w_{\std})|_{P_{\bK}^{-1}(c)} = 0$. 
Therefore symplectic reduction gives a symplectic manifold $(M_{\Delta}, \w_{\Delta})$ where
\begin{equation}\label{e:Mdelta}
	M_{\Delta} := P_{\bK}^{-1}(c)/\bK \quad\mbox{and the symplectic form $\w_{\Delta}$ is induced from $\w_{\std}|_{P_{\bK}^{-1}(c)}$}\,.
\end{equation}
It follows from Delzant's theorem \cite{Del88H} that $(M_{\Delta}, \w_{\Delta})$ and $(M, \w)$ are equivariantly symplectomorphic
as toric manifolds.

The following lemma that shows the significance of the assumption that $\Delta$ is an even moment polytope.
\begin{lemma}\label{l:Even}
	Let $\tau \in \T^{d}$ be the element such that $\tau \cdot z = -z$ for $z \in \C^{d}$.
	The torus $\bK$ from \eqref{e:defK} contains the element $\tau$ if and only if $\Delta$ is even.
\end{lemma}
\begin{proof}
	Note that $\tau = [\tfrac{1}{2}, \dots, \tfrac{1}{2}]$ in $\T^d = \R^d / \Z^d$ and therefore 
	since $\T = \mf{t}/\mf{t}_{\Z}$ it is clear from
	\eqref{e:defK} that $\tau \in \bK$ if and only if $\sum_{j=1}^{d} \tfrac{1}{2} \nu_{j} \in \mf{t}_{\Z}$.
\end{proof}

%%%%%%%%%%%%%%%%%%%%%%%%%%%%%

%%%%%%%%%%%%%%%%%%%%%%%%%%%%%
\subsubsection{The contact manifold $(\wh{M}, \xi)$ from Delzant's construction of $(M, \w)$}

Using Delzant's construction we will now describe the contact manifold $(\wh{M}, \xi)$ associated to an even
monotone symplectic toric manifold with moment polytope \eqref{e:Polytope1}.  
Define
$$
	\mf{k}_0 := \ker(c\fc \mf{k} \to \R)
$$
to be the annihilator of the linear functional $c = \iota^{*}(1, \dots, 1) \in \mf{k}^{*}$ from \eqref{e:LevelSet}
and define
\begin{equation}\label{e:defK0}
\mbox{$\bK_{0} \leq \bK$ to be the connected codimension $1$ subtorus with $\Lie(\bK_0) = \mf{k}_0$}.
\end{equation}
Since $\Delta$ in \eqref{e:Polytope1} is an even moment polytope, by Lemma~\ref{l:Even} we know $\bK_{0} + \ip{\tau} \leq \bK$,
where $\ip{\tau} \leq \bK$ is the subgroup generated by $\tau$.
Therefore $\bK_{0} + \ip{\tau}$ also acts freely on the level set $P_{\bK}^{-1}(c)$ from \eqref{e:LevelSet}. 

The contact manifold $(\wh{M}, \xi = \ker \alpha)$ is given by
\begin{equation}\label{e:WideHatM}
	\wh{M} := P_{\bK}^{-1}(c)/(\bK_{0} + \ip{\tau})
\end{equation}
and the contact form $\alpha$, which is induced from $\alpha_{\std}|_{P_\K^{-1}(c)}$,
is well-defined because the infinitesimal action of $\K_0$ is tangent to $\ker \alpha_{\std}$ along $P_\K^{-1}(c)$, which follows from \eqref{e:AlphaStuff}.  For the circle $S^{1} = \bK/(\bK_{0} + \ip{\tau})$, the natural projection map
\begin{equation}\label{e:BuildPQ}
	\pi\fc (\wh{M}, \a) \to (M_{\Delta}, \w_{\Delta})
\end{equation}
defines a principal $S^{1}$-bundle and satisfies $\pi^{*}\w_{\Delta} = d\alpha$ since $\w_{\std} = d\alpha_{\std}$.
Therefore by using a symplectomorphism $(M_\Delta,\omega_\Delta) \simeq (M,\omega)$,
we have that \eqref{e:BuildPQ} is the desired prequantization in Theorem~\ref{thm_main}.

We will now present a formula for the period of the the Reeb vector field of $(\wh{M}, \alpha)$
and hence the Euler class $e(\pi) \in H^{2}(M; \Z)$ of the principal $S^{1}$-bundle
\eqref{e:BuildPQ}.  For the functional $c\fc \mf{k} \to \R$ from \eqref{e:LevelSet}, let
$$
	c_{\mf{k}} \in \Z
$$
be the positive generator of the image $c(\mf{k}_{\Z}) \subset \Z$ of the integer lattice $\mf{k}_{\Z}:= \mf{k} \cap \Z^{d}$
and let
\begin{equation}\label{e:delta}
	\delta :=
\begin{cases}
1& \mbox{if }\tau \in \bK_{0}\\
2& \mbox{if }\tau \not\in \bK_{0}
\end{cases}.
\end{equation}

\begin{prop}\label{p:Euler}
The Reeb vector field for $(\wh{M}, \alpha)$ has period
$\frac{c_{\mf{k}}}{\delta}$ and the Euler class of $\pi \fc (\wh{M}, \alpha) \to (M, \w)$ equals
$$
e(\pi)=\, - \frac {\delta} {c_\mfk}\,c_1(M) \in H^{2}(M; \Z).
$$
\end{prop}
\begin{proof}
	Recall for a principal $S^{1}$-bundle $\pi\fc V \to M$ that 
	if $\alpha$ is a connection $1$-form on $V$ and $\w = \pi_{*}(d\alpha)$ is the curvature $2$-form on $M$,
	then the Euler class is given by \cite[Section 6.2(d)]{Mor01G}
	$$
		e(\pi):= \frac{-1}{\int_{\pi^{-1}(m)} \alpha}\, [\w] \in H^{2}(M; \Z)\,,
	$$
	the negative of the curvature form divided by the integral of the connection form over a fiber.
	For our prequantization $\int_{\pi^{-1}(m)} \alpha$ is the period of the Reeb vector field
	and we used the normalization $[\w] = c_{1}(M)$, so it suffices to compute that
	$\int_{\pi^{-1}(m)} \alpha = \frac{c_{\mf{k}}}{\delta}$.
	
	Consider $\wh M_\Delta :=P_\K^{-1}(c)/\K_0$ with $1$-form $\alpha_{\Delta}$
	induced from $\alpha_{\std}$.  Under the identification
	$(M, \w) \simeq (M_{\Delta}, \w_{\Delta})$
	from \eqref{e:Mdelta}, the projection map defines a prequantization
	$$
		\pi_{\Delta}\fc (\wh{M}_{\Delta}, \alpha_{\Delta}) \to (M, \w).
	$$
	We will compute that $\int_{\pi_{\Delta}^{-1}(\bar{z})} \alpha_{\Delta} = c_{\mf{k}}$
	for any $\bar{z} \in M$, which will suffice since $\wh{M}_{\Delta} \to \wh{M}$ is a degree $\delta$ cover.
	By the definition of $(\wh{M}_{\Delta}, \alpha_{\Delta})$, its Reeb vector field can be
	represented by the infinitesimal action of
	$X_{\lambda}$ from \eqref{e:Xvector} on $P^{-1}_{\bK}(c)$ for any 
	\begin{equation}\label{e:LambdaK1}
		\lambda \in \mf{k} \quad\mbox{such that}\quad \ip{\lambda, c} = 1.
	\end{equation}
	For any such $\lambda$, the period of the Reeb orbit can also be characterized as the smallest
	$T > 0$ so that $\exp(T\lambda) \in \bK_{0}$ for the exponential map $\exp\fc \mf{k} \to \bK$.
	Since $\ip{\lambda_0, c} = 0$ for any $\lambda_0 \in \mf{k}_0$ and $\bK_0$ is connected, we can choose $\lambda$
	as in \eqref{e:LambdaK1} so that the first return is at $\exp(T\lambda) = 1 \in \bK$.
	In this case $T\lambda \in \mf{k}_{\Z}$ and $T = \ip{T\lambda, c} \in c(\mf{k}_{\Z})$, so
	therefore $T = c_{\mf{k}}$ the minimal positive generator of $c(\mf{k}_{\Z})$.
\end{proof}

\begin{rem}
	Both options in \eqref{e:delta} actually occur.  For the case of $\CP^{n}$ we have $\tau \notin \bK_0$,
	since $\bK_0 = 1$, and for $\CP^n \times \CP^n$ below we do have $\tau \in \bK_0$.
\end{rem}

%%%%%%%%%%%%%%%%%%%%%%%%

%%%%%%%%%%%%%%%%%%%%%%%%
\subsubsection{An example in the case $M = \CP^{n-1} \times \CP^{n-1}$}\label{s:Example}

Consider the even toric monotone symplectic manifold 
$$(M, \w)=(\C P^{n-1} \times \C P^{n-1}, n\sigma \oplus n\sigma)$$
where $\int_{\CP^1}\sigma = 1$.  Its moment polytope is
$$
	\left\{ \textstyle (x, x') \in (\R^{2n-2})^* \mid x_j + 1 \geq 0\,,\,\, -\sum_{j=1}^{n-1} x_j + 1 \geq 0\,,\,\,
	x'_j + 1 \geq 0\,,\,\, -\sum_{j=1}^{n-1} x'_j + 1 \geq 0
	\right\}
$$
where we have identified $\T = \R^{2n-2}/\Z^{2n-2}$.  In this case $\bK_{0} \leq \bK$ are the subtori of $\T^{2n}$
whose Lie algebras in $\R^{2n}$ have bases
$$
\mf{k}_{0} = \Span\left\{\textstyle \sum_{j=1}^{n} \epsilon_{j} - \sum_{j=1}^{n} \epsilon_{j}'\right\}
\quad\mbox{and}\quad
\mf{k} = \Span\left\{\textstyle\sum_{j=1}^{n} \epsilon_{j}\,,\,\, \sum_{j=1}^{n} \epsilon_{j}'\right\}\,.
$$
The moment map $P_{\bK}\fc \C^{2n} \to \mf{k}^{*} = (\R^{2})^{*}$ for the action of $\bK$ on $\C^{2n}$ is
$$
	P_{\bK}(z,z') = \pi \left( \textstyle\sum_{j=1}^{n} \abs{z_j}^2\,,\,\,\sum_{j=1}^{n} |z_{j}'|^2\right)
$$
and we have $P^{-1}_{\bK}(c)$ is $S^{2n-1} \times S^{2n-1} \subset \C^{2n}$ since
$$
	P_\K^{-1}(c) = \left\{\textstyle(z, z')\in\C^{2n}\,|\,\pi\sum_{j=1}^{n+1} \abs{z_j}^2 = \pi\sum_{j=1}^{n+1} |z_{j}'|^2 = n\right\}.
$$
The action of the circle $\bK_{0}$ on $\C^{2n}$ is given by 
$\zeta \cdot (z,z') = (\zeta z, \ol{\zeta} z')$ for $\zeta \in S^{1}$ the unit circle
and note that $\tau \in \bK_{0}$.  
The contact manifold is
\begin{equation}\label{e:WideMex}
	\wh{M} =  (S^{2n-1}\times S^{2n-1})/\bK_{0} \quad\mbox{with contact form $\alpha$ induced by} \quad 
	\alpha_{\std}|_{S^{2n-1}} \oplus \alpha_{\std}|_{S^{2n-1}}\,.
\end{equation}
The Reeb vector field $R_{\alpha}$ is represented by $X_{\lambda}$ with $\lambda = \tfrac{1}{2n}(1,\dots,1) \in \R^{2n}$ from \eqref{e:Xvector} and it has period is $c_{\mf{k}} = n$, so therefore the prequantization is the $\R/n\Z$-bundle
\begin{equation}\label{e:PreEx}
	\pi\fc \big(\wh{M}, \alpha\big) \to (\CP^{n-1} \times \CP^{n-1}, n\sigma \oplus n\sigma).
\end{equation}
Since the first Chern class $c_{1}(\CP^{n-1} \times \CP^{n-1}) = (n,n) \in H^{2}(\CP^{n-1} \times \CP^{n-1}; \Z)$,
we have
$$
	e(\pi) = (-1,-1) \in H^{2}(\CP^{n-1} \times \CP^{n-1}; \Z)
$$
from Proposition~\ref{p:Euler}.  

Rescaling so that prequantization is a $\R/\Z$-bundle, we see that 
$$\pi\fc \wh{M} \to (\CP^{n-1} \times \CP^{n-1}, \sigma \oplus \sigma)$$
is the standard Boothby--Wang prequantization \cite{BooWan58O}.
For the case of $n=2$, it is known that $\wh{M}$ is contactomorphic to $UT^{*}S^3$ the unit cotangent bundle of $S^3$, for instance see
\cite[Section 6.1]{AbrMac12C}.  However when $n \geq 3$, it follows from \cite[Theorem 8]{BooWan58O} that $\wh{M}$ is not 
even topologically a unit cotangent bundle.
%%%%%%%%%%%%%%%%%%%%%%%%%%%%%
%%%%%%%%%%%%%%%%%%%%%%%%%%%%%

%%%%%%%%%%%%%%%%%%%%%%%%%%%%%
%%%%%%%%%%%%%%%%%%%%%%%%%%%%%
\subsection{Applying Theorem~\ref{thm_ART} to prove Theorem~\ref{thm_main}}\label{s:ApplyingART}

\begin{proof}[Proof of Theorem~\ref{thm_main}]
Theorem~\ref{thm_main} will be proved by applying Theorem~\ref{thm_ART} to Givental's
quasi-morphism $\mu_{\Giv} \fc \wt\Cont_{0}(\RP^{2d-1}_{\g}) \to \R$ in the setting
\begin{equation}\label{e:SetUpMainThm}
	(\RP^{2d-1}_{\g}, \xi_{\g}, \alpha_{\std}) \supset (Y, \alpha_{\std}|_{Y}) \stackrel{\rho}{\longrightarrow} (\wh{M}, \xi, \a)
\end{equation}
for an appropriate $\g$ and $Y$ that we describe below. 

By \cite[Lemma 4.8]{AbrMac13R} there is a primitive vector $\g = (\g_{1}, \ldots, \g_{d}) \in \mfk_\Z := \Z^{d} \cap \mf{k}$
where each $\g_{j} \geq 1$.  Fix such a $\g$ and consider the sphere $S^{2d-1}_{\g}$ from \eqref{e:SphereGamma}.
Note that
$$
	P_{\bK}^{-1}(c) \subset S^{2d-1}_{\g} = \big\{z \in \C^{d} \mid \ip{\g, P}(z) = \ip{\g, c} \big\}
$$
since $z \in P_{\bK}^{-1}(c)$ is equivalent to $\ip{\lambda, P}(z) = \ip{\lambda, c}$ for all $\lambda \in \mf{k}$.  
Since $\Delta$ is even we know $\tau \in \bK$ by Lemma~\ref{l:Even} and modding out by the antipodal 
$\Z_{2} = \ip{\tau}$ action gives the submanifold
\begin{equation}\label{e:definitionOfY}
	Y:= P_{\bK}^{-1}(c)/\ip{\tau} \subset (\RP^{2d-1}_{\g}\!, \xi_{\g})
\end{equation}
where recall $(\RP^{2d-1}_{\g}\!, \xi_{\g}) = (S^{2d-1}_{\g}/\ip{\tau}, \ker \alpha_{\std})$.
The natural projection map
\begin{equation}\label{e:DefinitionOfRho}
	\rho\fc Y \to \wh{M}
\end{equation}
is a principal $(\bK_{0} + \ip{\tau})/\ip{\tau}$-bundle and by the construction of the $1$-form $\alpha$ from \eqref{e:WideHatM}
it follows that $\rho^{*}\alpha = \alpha_{\std}|_{Y}$.

To verify the geometric setting \eqref{e:SetUpART} of Theorem~\ref{thm_ART} it remains to prove
$Y \subset \RP^{2d-1}_{\g}$ is strictly coisotropic with respect to $\alpha_{\std}$. 
Note that $P_{\bK}^{-1}(c) \subset (\C^{d}, \w_{\std})$ is a coisotropic submanifold, meaning for all $z \in P_{\bK}^{-1}(c)$:
\begin{equation}\label{e:LevelCoisotropic}
	(T_{z}P_{\bK}^{-1}(c))^{\w_{\std}} := 
	\{X \in T_{z}\C^{d} \mid \iota_{X}\w_{\std} = 0 \mbox{ on $T_{z}P_{\bK}^{-1}(c)$}\} \subset T_{z}P_{\bK}^{-1}(c)\,,
\end{equation}
since $P_{\bK}^{-1}(c)$ is the regular level set of a moment map or as can be verified with \eqref{e:AlphaStuff}.
It now follows from \eqref{e:LevelCoisotropic} and $d\alpha_{\std} = \w_{\std}$ that $P_{\bK}^{-1}(c) \subset S^{2d-1}_{\g}$
satisfies the condition \eqref{e:StrictCoisotropic} to be strictly coisotropic with respect to $\alpha_{\std}$ and
therefore so is  $Y \subset \RP^{2d-1}_{\g}$.

Using the definition \eqref{e:LevelSet} of $P_{\bK}^{-1}(c)$ we know
$$
 	\{\abs{z_{1}}^{2} = \dots = \abs{z_{d}}^{2} = 1/\pi\} \subset P_{\bK}^{-1}(c)
$$
since if $z \in \{\abs{z_{1}}^{2} = \dots = \abs{z_{d}}^{2} = 1/\pi\}$, then for any $\lambda \in \mf{k}$ one has 
\begin{equation}\label{e:KeyMonotonicity}
	\ip{\lambda, P}(z) = \pi \sum_{j=1}^{d}\lambda_{j}\tfrac{1}{\pi} = \sum_{j=1}^{d} \lambda_{j} = \ip{\lambda, c}
\end{equation}
where the last equality follows from the fact that $c:= \iota^{*}(1, \dots, 1)$.
Hence $T_{\RP} \subset Y$ for the torus $T_{\RP} \subset \RP^{2d-1}_{\g}$ from \eqref{e:Clifford}, 
which is $\mu_{\Giv}$-superheavy by Lemma~\ref{exam_torus_RP_superheavy}.  Therefore by Theorem~\ref{t:SHandsubheavy}
we know $Y \subset \RP^{2d-1}_{\g}$ is $\mu_{\Giv}$-subheavy.
Applying Theorem~\ref{thm_ART} to \eqref{e:SetUpMainThm} constructs the desired
monotone quasi-morphism $\overline{\mu}\fc \wt{\Cont}_{0}(\wh{M}, \xi) \to \R$ 
with the vanishing property.
\end{proof}

\begin{rem}
	For any closed even symplectic toric manifold $(M, \w, \T)$
	the construction in this section can be modified to produce a prequantization $\pi\fc (\wh{M}, \alpha) \to (M, \w)$ that is 
	constructed by contact reduction of a real projective space.  
	Without the monotonicity assumption however, one needs to replace $c = \iota^{*}(1, \dots, 1)$ with
	$\iota^{*}(a_1, \dots, a_d)$ where the $a_j$ are the support constants in the moment polytope \eqref{e:Polytope}.
	With this change \eqref{e:KeyMonotonicity} no longer holds so the reduction will not pass through the 
	superheavy torus $T_{\RP}$.  This is similar to the proof of \cite[Proposition 4.9]{AbrMac13R}.  
\end{rem}

%%%%%%%%%%%%%%%%%%%%%%%%%%%%%%%%%%%%%%%%%%%%%%%%%%%%%%%%%%%%

%%%%%%%%%%%%%%%%%%%%%%%%%%%%%%%%%%%%%%%%%%%%%%%%%%%%%%%%%%
%%%%%%%%%%%%%%%%%%%%%%%%%%%%%%%%%%%%%%%%%%%%%%%%%%%%%%%%%%

%%%%%%%%%%%%%%%%%%%%%%%%%%%%%%%%%%%%%%%%%%%%%%%%%%%%%%%%%%
%%%%%%%%%%%%%%%%%%%%%%%%%%%%%%%%%%%%%%%%%%%%%%%%%%%%%%%%%%
\section{Proof of the reduction theorem for quasi-morphisms}\label{s:ProofOfART}

In this section we will present the proof of Theorem~\ref{thm_ART}.

%%%%%%%%%%%%%%%%%%%%%%%%%%%%%%%%%%%%%%%%%%%%%%%%%%%%%%%%%%
\subsection{Preliminary lemmas}

%%%%%%%%%%%%%%%%%%%%%%%%%%%%%
\subsubsection{Geometric setting of Theorem~\ref{thm_ART}}

Let us begin by collecting a few lemmas about the geometric setting of Theorem~\ref{thm_ART},
\begin{equation}\label{e:SetUpART1}
	(V, \xi, \alpha) \supset (Y, \alpha|_{Y}) \stackrel{\rho}{\longrightarrow} (\ol{V}, \ol{\xi}, \ol{\alpha})
\end{equation}
where $(V, \xi, \alpha)$ and $(\ol{V}, \ol{\xi}, \ol{\alpha})$ are closed contact manifolds,
$Y \subset V$ is a closed submanifold that is strictly coisotropic with respect to $\alpha$,
and $\rho$ is a fiber bundle such that $\rho^{*}\ol{\alpha} = \alpha|_{Y}$.
\begin{lemma}\label{l:Reeb}
	The map $d\rho\fc TY \to T\ol{V}$ relates the Reeb vector fields: $d\rho \circ R_{\a}|_{Y} = R_{\ol{\a}} \circ \rho$.
\end{lemma}
\begin{proof}
	Note that by Definition~\ref{d:StrictlyCoiso} of strictly coisotropic the Reeb vector field $R_{\alpha}|_{Y}$
	is tangent to $Y$.  To show $d\rho \circ R_{\a} = R_{\ol{\a}} \circ \rho$ one computes
	$$
		\ol{\a}(d\rho(R_{\a})) = \rho^{*}\ol{\a}(R_{\a}) = \a(R_{\a}) = 1
	$$
	and for any $u \in TY$ one has
	$$
		d\ol{\a}(d\rho(R_{\a}), d\rho(u)) = d(\rho^{*}\ol{\a})(R_{\a}, u) = d\alpha(R_{\a}, u) = 0
	$$
	which proves $\iota_{d\rho(R_{\a})}d\ol{\a} = 0$ since $d\rho \fc TY \to T\ol{V}$ is surjective.
	By definition of $R_{\ol{\alpha}}$ this proves $d\rho\circ R_{\a}|_{Y} = R_{\ol{\a}} \circ \rho$.
\end{proof}

\noindent
For an $\ol{h} \in C^{\infty}([0,1] \times \ol{V})$, an \textbf{extension} of $\ol{h}$ will be any $h \in C^{\infty}([0,1] \times V)$ such that
$$h|_{[0,1] \times Y} = \rho^{*}\ol{h}.$$

\begin{lemma}\label{l:extension}
	If $h \in C^{\infty}([0,1] \times V)$ is an extension of $\ol{h} \in C^{\infty}([0,1] \times \ol{V})$, then
	\begin{enumerate}
	\item the contact vector field $X_{h_{t}}|_{Y}$ is tangent to $Y$,
	\item the contact vector fields of $h,\ol h$ are related by $d\rho$: $d\rho \circ X_{h_{t}}|_{Y} = X_{\ol{h}_{t}} \circ \rho$, and
	\item as maps $\rho \circ \phi_{h}^{t}|_{Y} = \phi_{\ol{h}}^{t} \circ \rho \fc Y \to \ol{V}$ 
	and $\rho \circ (\phi_{h}^{t})^{-1}|_{Y} = (\phi_{\ol{h}}^{t})^{-1} \circ \rho \fc Y \to \ol{V}$ for all $t \in [0,1]$.
	\end{enumerate}
\end{lemma}
\begin{proof}	
	It suffices to prove (i) for autonomous $h \in C^{\infty}(V)$ and $\ol{h} \in C^{\infty}(\ol{V})$.
	Let $u \in TY$,
	then by the definition of $X_{h}$ from \eqref{e:ContactHam}
	and the relations
	$$
	\mbox{$\rho^{*}\ol{\a} = \a|_{Y}$\,,\,\, $\rho^{*}\ol{h} = h|_{Y}$\,,\,\, $d\rho \circ R_{\a}|_{Y} = R_{\ol{\a}} \circ \rho$}
	$$
	we have
	$$
		d\a(X_{h}, u) = -dh(u) + dh(R_{\a})\a(u)
		= -d\ol{h}(d\rho(u)) + d\ol{h}(R_{\ol{\a}})\, \ol{\a}(d\rho(u))
		= d\ol{\a}(X_{\ol{h}}, d\rho(u)).
	$$
	Since $X_{\ol{h}} = d\rho(v)$ for some $v \in TY$, taking any $u \in (TY)^{d\alpha} \subset TY$ (see \eqref{e:StrictCoisotropic}) we get
	$$
		d\alpha(X_{h}, u) = d\ol{\a}(d\rho(v), d\rho(u)) = d\alpha(v, u) = 0\,, 
	$$
	and hence $X_{h}|_{Y} \in ((TY)^{d\alpha})^{d\alpha}$.  Since $Y$ is strictly coisotropic,
	it follows from \eqref{e:StrictCoisotropic} that $((TY)^{d\alpha})^{d\alpha} =TY$ and hence $X_{h}|_{Y} \in TY$.

	For item (ii), similar considerations as above show
	$$
		\ol{\a}(d\rho(X_{h})) = \a(X_{h}) = \ol{h} \quad\mbox{at the point $\rho(y)$}
	$$ 
	and likewise for any $u \in TY$
	$$
		d\ol{\a}(d\rho(X_{h}), d\rho(u)) = d\a(X_{h}, u) = -dh(u) + dh(R_{\a})\a(u) =
		-d\ol{h}(d\rho(u)) + d\ol{h}(R_{\ol{\a}}) \ol{\a}(d\rho(u))\,.
	$$
	Since $d\rho\fc TY \to T\ol{V}$ is surjective this shows $d\rho \circ X_{h}|_{Y} = X_{\ol{h}} \circ \rho$.
	
	The first part of item (iii) immediately follows from item (ii).  The second part of item (iii) follows 
	from the first part via $\rho \circ (\phi_{h}^{t})^{-1}|_Y = (\phi_{\ol{h}}^{t})^{-1} \circ \phi_{\ol{h}}^{t} \circ \rho \circ (\phi_{h}^{t})^{-1}|_Y
	= (\phi_{\ol{h}}^{t})^{-1} \circ \rho$.
\end{proof}
%%%%%%%%%%%%%%%%%%%%%%%%%%%%%

%%%%%%%%%%%%%%%%%%%%%%%%%%%%%
\subsubsection{Using the subheavy assumption}

Recall that in Theorem~\ref{thm_ART} that in addition to the geometric setting in \eqref{e:SetUpART1}, $Y \subset V$ is
subheavy with respect to a monotone quasi-morphism $\mu \fc \wt{\Cont}_{0}(V) \to \R$.

\begin{lemma}\label{l:CommuteRho}
	If $h \in C^{\infty}([0,1] \times V)$ is such that $X_{h_{t}}|_{Y} \in TY$ for all $t \in [0,1]$ and
	\begin{equation}\label{e:commute}
	\mbox{$\rho \circ \phi_{h}^{t}|_{Y} = \rho \fc Y \to \ol{V}$ \quad for all $t \in [0,1]$,}
	\end{equation}
	then 	$\mu(\wt\phi_{h}) = 0$.
\end{lemma}
\begin{proof}
	Differentiating \eqref{e:commute} with respect to $t$ gives $d\rho(X_{h_{t}}) = 0$
	and hence on $Y$:
	$$
		h_{t} = \a(X_{h_{t}}) = \ol{\a}(d\rho(X_{h_{t}})) = 0.
	$$
	Now pick autonomous Hamiltonians $g, k \in C^{\infty}(V)$ so that $g \leq h \leq k$ and $g|_{Y} = k|_{Y} = 0$.
	Since $Y$ is $\mu$-subheavy it follows that $\mu(\wt\phi_{g}) = \mu(\wt\phi_{k}) = 0$ and therefore
	$\mu(\wt\phi_{h}) = 0$ since $\mu$ is monotone.
\end{proof}

\begin{lemma}\label{l:loops}
	Let $\ol{h} \in C^{\infty}([0,1] \times \ol{V})$ be such that $\{\phi_{\ol{h}}^{t}\}_{t \in [0,1]}$ is a contractible loop
	in $\Cont_{0}(\ol{V})$.  For any contact Hamiltonian $h \in C^{\infty}([0,1] \times V)$ 
	that is an extension of $\ol{h}$ 	it follows $\mu(\wt\phi_{h}) = 0$.
\end{lemma}
\begin{proof}
	Let $\Phi\fc [0,1]^{2} \to \Cont_{0}(\ol{V})$ be a null-homotopy of loops for $\{\phi_{\ol{h}}^{t}\}$, 
	so $\{\Phi^{s}_{t}\}_{t \in [0,1]}$ is a loop of contactomorphisms of $\ol{V}$ for fixed $s$ where
	$\Phi^{0}_{t} = \id_{\ol{V}}$ and $\Phi^{1}_{t} = \phi_{\ol{h}}^{t}$.  
	Let 
	$$\ol{H}^{s}_{\cdot} \fc [0,1] \times \ol{V} \to \R$$
	 be the contact Hamiltonian generating the contact
	isotopy $\{\Phi_{t}^{s}\}_{t \in [0,1]}$ for fixed $s$ and note $\ol{H}^{1}_{t} = \ol{h}_{t}$.
	
	Let $H\fc [0,1]^{2} \times V \to \R$ be an extension of $\ol{H}$ so that $H^{s}_{t}|_{Y} = \ol{H}^{s}_{t} \circ \rho$
	and $H^{1}_{t} = h_{t}$ is the chosen extension of $\ol{h}$.
	Let $\{\Psi^{s}_{t}\}_{t \in [0,1]}$ be the contact isotopy of $V$ generated by
	$$H^{s}_{\cdot}\fc [0,1] \times V \to \R$$ for fixed $s$.
	It follows from Lemma~\ref{l:extension} that $\rho \circ \Psi_{t}^{s}|_{Y} = \Phi_{t}^{s} \circ \rho \fc Y \to \ol{V}$ for 
	all $s$ and $t$.  In particular the concatenation of paths
	$$
		\{\Psi_{0}^{1-u}\}_{u \in [0,1]}\# \{\Psi_{u}^{0}\}_{u \in [0,1]} \# \{\Psi_{1}^{u}\}_{u \in [0,1]}
	$$
	defines an isotopy $\{\psi^{u}\}_{u \in [0,1]}$ in $\Cont_{0}(V)$ such that $\psi^{u}(Y) = Y$
	and $\rho \circ \psi^{u}|_{Y} = \id_{\ol{V}} \circ \rho$ since
	$$
		\id_{\ol{V}} = \{\Phi_{0}^{1-u}\}_{u \in [0,1]}\# \{\Phi_{u}^{0}\}_{u \in [0,1]} \# \{\Phi_{1}^{u}\}_{u \in [0,1]}
		\quad\mbox{in $\Cont_{0}(\ol{V})$.}
	$$
	Let $\wt\psi\in\wt\Cont_0(V)$ be the element represented by $\{\psi^u\}_{u\in[0,1]}$. By Lemma~\ref{l:CommuteRho} we know 
	$\mu(\wt\psi) = 0$
	and hence $\mu(\wt\phi_{h}) = \mu(\wt\psi) = 0$ since $\{\psi^{u}\}_{u \in [0,1]}$ is homotopic with fixed endpoints
	to the isotopy $\{\Psi^{1}_{t}\}_{t\in [0,1]} = \{\phi_{h}^{t}\}_{t \in [0,1]}$.
\end{proof}
%%%%%%%%%%%%%%%%%%%%%%%%%%%%%

%%%%%%%%%%%%%%%%%%%%%%%%%%%%%%%%%%%%%%%%%%%%%%%%%%%%%%%%%%

%%%%%%%%%%%%%%%%%%%%%%%%%%%%%%%%%%%%%%%%%%%%%%%%%%%%%%%%%%
\subsection{Proof of Theorem~\ref{thm_ART}}
Let $\cP\!\Cont_{0}(\ol{V})$ denote the group of contact isotopies of $(\ol{V}, \ol{\xi})$ based at the identity with time-wise composition
as the product: $\{\phi_{\ol{h}}^{t}\} * \{\phi_{\ol{k}}^{t}\} = \{\phi_{\ol{h}}^{t} \circ \phi_{\ol{k}}^{t}\}$.
Passing
to homotopy classes of isotopies with fixed endpoints is a group homomorphism
$\cP\!\Cont_{0}(\ol{V}) \to \wt\Cont_{0}(\ol{V})$ whose kernel consists of contractible loops based at the identity.

\begin{proof}[Proof of Theorem~\ref{thm_ART}]
	We will break the proof of Theorem~\ref{thm_ART} into a few steps.
	
	\textit{Independence of choice of extension:}  
	We will first prove
	\begin{equation}\label{e:QmIsotopy}
		\ol{\mu} \fc \cP\!\Cont_{0}(\ol{V}) \to \R \quad\mbox{defined by}\quad \ol\mu(\{\phi_{\ol{h}}^{t}\}) = \mu(\wt\phi_{h})
	\end{equation}
	where $h \in C^{\infty}([0,1] \times V)$ is any extension of $\ol{h} \in C^{\infty}([0,1] \times \ol{V})$ is well-defined.
	So let $h$ and $k$ both be extensions of $\ol{h}$. 
	For any positive integer $m$ by Lemma~\ref{l:extension}(iii) it follows
	$$
		\rho \circ (\phi_{h}^{t})^{m} \circ (\phi_{k}^{t})^{-m}|_{Y} = (\phi_{\ol{h}}^{t})^{m} \circ (\phi_{\ol{k}}^{t})^{-m} \circ \rho = \rho
	$$
	and hence $\mu(\wt\phi_{h}^{m}\wt\phi_{k}^{-m}) = 0$ by Lemma~\ref{l:CommuteRho}.
	Using that $\mu$ is a homogeneous quasi-morphism \eqref{e:Quasimorphism} we have
	$$
		\big|\mu(\wt\phi_{h}) - \mu(\wt\phi_{k})\big| = \tfrac{1}{m}\big|\mu(\wt\phi_{h}^{m}) - \mu(\wt\phi_{k}^{m})\big|
		\leq \tfrac{1}{m}(D + \mu(\wt\phi_{h}^{m}\wt\phi_{k}^{-m})) = \tfrac{D}{m}
	$$
	and taking the limit as $m \to \infty$ shows $\mu(\wt\phi_{h}) = \mu(\wt\phi_{k})$.
	
	\textit{Homogeneous quasi-morphism on $\cP\!\Cont_{0}(\ol{V})$:}
	We will first show that \eqref{e:QmIsotopy} defines a quasi-morphism.
	Let $g, h, k \in C^{\infty}([0,1] \times V)$
	be extensions of $\ol g, \ol{h}, \ol{k} \in C^{\infty}([0,1] \times \ol{V})$ 
	where $\ol{g}$ generates the product of $\ol{h}$ and $\ol{k}$, that is 
	$\phi_{\ol{g}}^{t} = \phi_{\ol{h}}^{t} \circ \phi_{\ol{k}}^{t}$.  By Lemma~\ref{l:extension}(iii) we have
	$$
		\rho \circ (\phi_{g}^{t})^{-1} \circ \phi_{h}^{t} \circ \phi_{k}^{t}|_{Y} = 
		(\phi_{\ol{g}}^{t})^{-1} \circ \phi_{\ol{h}}^{t} \circ \phi_{\ol{k}}^{t} \circ \rho = \rho
	$$
	so by Lemma~\ref{l:CommuteRho} it follows $\mu(\wt\phi_{g}^{-1} \wt\phi_{h} \wt\phi_{k}) = 0$.
	If $D$ is as in \eqref{e:Quasimorphism} for $\mu$, it follows
	$$
		\big|\ol{\mu}(\{\phi_{\ol{h}}^{t}\}*\{\phi_{\ol{k}}^{t}\}) - \ol{\mu}(\{\phi_{\ol{h}}^{t}\}) - \ol{\mu}(\{\phi_{\ol{k}}^{t}\})\big|
		= \big|\mu(\wt\phi_{g}) - \mu(\wt\phi_{h}) - \mu(\wt\phi_{k})\big| \leq 2D
	$$
	and therefore $\ol{\mu}$ in \eqref{e:QmIsotopy} is a quasi-morphism.
	
	We will now show that $\ol{\mu}$ is homogeneous. 
	If $h, g^{(m)} \in C^{\infty}([0,1] \times V)$ are extensions of $\ol{h}, \ol{g} \in C^{\infty}([0,1] \times \ol{V})$
	where $\phi_{\ol{g}}^{t} = (\phi_{\ol{h}}^{t})^{m}$ for an integer $m \in \Z$, 
	then again Lemma~\ref{l:CommuteRho} shows $\mu(\wt\phi_{g^{(m)}}^{-1} \wt\phi_{h}^{m}) = 0$.  It follows
	\begin{equation}\label{e:HomoBound}
		\big|\ol{\mu}(\{\phi_{\ol{h}}^{t}\}^{m}) - m\,\ol{\mu}(\{\phi_{\ol{h}}^{t}\})\big| =
		\big|\mu(\wt\phi_{g^{(m)}}) - \mu(\wt\phi_{h}^{m})\big|
		\leq D + \big|\mu(\wt\phi_{g^{(m)}}^{-1} \wt\phi_{h}^{m})\big| = D\,.
	\end{equation}
	Dividing \eqref{e:HomoBound} by $m$ and taking $m \to \infty$ shows that
	\begin{equation}\label{e:Homo}
		\lim_{m \to \infty}\frac{\ol\mu(\{\phi^t_{\ol h}\}^m)}{m} = \ol\mu(\{\phi^t_{\ol h}\})\,,
	\end{equation}
	and it follows from \eqref{e:Homo} that $\ol{\mu}$ is homogeneous, see for instance \cite[Lemma 2.21]{Cal09S}.
	
	\textit{Descent to a quasi-morphism on $\wt\Cont_{0}(\ol{V})$:}
	Since $\ol{\mu}\fc \cP\!\Cont_{0}(\ol{V}) \to \R$ vanishes on the kernel of the map 
	$\cP\!\Cont_{0}(\ol{V}) \to \wt\Cont_{0}(\ol{V})$
	by Lemma~\ref{l:loops}, it follows that $\ol{\mu}$ in \eqref{e:QmIsotopy} descends to a well-defined homogeneous quasi-morphism $\ol{\mu} \fc \wt\Cont_{0}(\ol{V}) \to \R$ by \cite[Lemma 3.2]{Bor12S}.
	
	\textit{Non-zero:}
	To see the quasi-morphism $\ol{\mu} \fc \wt\Cont_{0}(\ol{V}) \to \R$ is not zero,
	let $\ol{h} \in C^{\infty}(\ol{V})$ be any positive contact Hamiltonian
	and pick $h \in C^{\infty}(V)$ to be a positive extension.
	Since $V$ is $\mu$-superheavy by Proposition~\ref{p:Basics}(iv) it follows that
	$\ol{\mu}(\wt\phi_{\ol{h}}) = \mu(\wt\phi_{h}) > 0$.
	
	\textit{Monotone:}
	If $\ol{h} \leq \ol{k}$, then one can pick extensions $h$ and $k$ such that $h \leq k$.  Since $\mu$ is
	monotone it follows $\mu(\wt\phi_{h}) \leq \mu(\wt\phi_{k})$ and therefore 
	$\ol{\mu}(\wt\phi_{\ol{h}}) \leq \ol{\mu}(\wt{\phi}_{\ol{k}})$
	by definition.  
	
	\textit{Vanishing:}
		Assume now that $\mu$ has the vanishing property.
		Let $U \subset \ol{V}$ be an open set that is displaceable by an element of $\Cont_{0}(\ol{V})$, then it follows from
		Lemma~\ref{l:extension}(iii) that an open neighborhood $N \subset V$ of $\rho^{-1}(U) \subset Y$ is displaceable 
		by an element of $\Cont_{0}(V)$.   
		Now if $\ol{h} \in C^{\infty}([0,1] \times \ol{V})$ has
		$\supp(\ol{h}) \subset [0,1] \times U$, then there is an extension $h \in C^{\infty}([0,1] \times V)$ of $\ol{h}$ with 
		$\supp(h) \subset [0,1] \times N$.  Since $\mu$ has the vanishing property it follows $\mu(\wt\phi_{h}) = 0$ and so by definition
		$\ol{\mu}(\wt\phi_{\ol{h}}) = 0$ as well.  Therefore $\ol{\mu}$ also has the vanishing property.
		
	\textit{$C^0$-continuity:}
	Assume that $\mu$ is $C^0$-continuous and let $\ol{h}^{(n)} \in C^\infty([0,1] \times \ol{V})$ be a sequence of
	contact Hamiltonians $C^0$-converging to $\ol{h} \in C^\infty([0,1] \times \ol{V})$.  Then we can pick 
	extensions $h^{(n)}$ and $h$ in $C^\infty([0,1] \times V)$ with $C^0$-convergence $h^{(n)} \to h$.
	Since $\mu$ is $C^0$-continuous, we have $\mu(\wt\phi_{h^{(n)}}) \to \mu(\wt\phi_h)$ and hence
	$\ol\mu(\wt\phi_{\ol{h}^{(n)}}) \to \ol\mu(\wt\phi_{h})$ as well. Therefore $\ol\mu$ is $C^0$-continuous.
\end{proof}
%%%%%%%%%%%%%%%%%%%%%%%%%%%%%%%%%%%%%%%%%%%%%%%%%%%%%%%%%%

%%%%%%%%%%%%%%%%%%%%%%%%%%%%%%%%%%%%%%%%%%%%%%%%%%%%%%%%%%
%%%%%%%%%%%%%%%%%%%%%%%%%%%%%%%%%%%%%%%%%%%%%%%%%%%%%%%%%%

%%%%%%%%%%%%%%%%%%%%%%%%%%%%%%%%%%%%%%%%%%%%%%%%%%%%%%%%%%
%%%%%%%%%%%%%%%%%%%%%%%%%%%%%%%%%%%%%%%%%%%%%%%%%%%%%%%%%%
\section{Proof of rigidity and vanishing results}

In this section we will present the remaining proofs.

%%%%%%%%%%%%%%%%%%%%%%%%%%%%%%%%%%%%%%%%%%%%%%%%%%%%%%%%%%
\subsection{Proof of rigidity results from Section~\ref{s:ContactRigidity}}\label{s:RigidityProof}

We will first prove the following lemma that shows that there is no difference between positive and negative in terms of defining a subset to be superheavy with respect to a quasi-morphism on $\wt\Cont_{0}(V)$.

\begin{lemma}\label{l:LessThanZero}
	If $\mu \fc \wt\Cont_{0}(V,\xi) \to \R$ is a monotone quasi-morphism and $Y \subset V$ is a closed subset,
	then $Y$ is $\mu$-superheavy if and only if 
	$\mu(\wt\phi_{h}) < 0$ for all autonomous contact Hamiltonians where $h|_{Y} < 0$.
\end{lemma}
\begin{proof}
	If $h$ is autonomous, then $\wt\phi_{h}^{-1}$ is generated by the contact Hamiltonian $-h$
	and therefore $\mu(\wt\phi_{-h}) = - \mu(\wt\phi_{h})$ since $\mu$ is homogeneous.
	The lemma now follows from the definition of $Y$ being superheavy from Definition~\ref{d:sh}.
\end{proof}

\noindent
Let us now prove Proposition~\ref{p:Basics} and
Theorem~\ref{t:SHandsubheavy}, detailing the basic properties of superheavy and subheavy sets.

\begin{proof}[Proof of Proposition~\ref{p:Basics}]
	To prove item (i) recall that any two contact forms $\alpha$ and $\alpha'$ for $\xi$ differ by multiplication
	by a positive function $f\fc V \to \R$: $\alpha' = f \alpha$.  
	If $h \in C^{\infty}([0,1] \times V)$ is the contact Hamiltonian associated to the contact isotopy
	$\{\phi^{t}\}_{t \in [0,1]}$ using the form $\alpha$, then $f\cdot h$ is the contact Hamiltonian associated to 
	the same isotopy using the
	form $\alpha'$.  Hence $h|_{[0,1] \times Y} > 0$ if and only if $f\cdot h|_{[0,1] \times Y} > 0$, so $\mu$-superheaviness
	is independent of contact form and likewise for $\mu$-subheaviness.
	
	Item (ii) is immediate since if $h|_{Y}> 0$, then $h|_{Z} > 0$ and hence $\mu(\wt\phi_{h}) > 0$ since
	$Z$ is $\mu$-superheavy.  The argument for $\mu$-subheaviness is analogous.
	
	 To prove (iii), first note homogeneous quasi-morphisms are conjugation-invariant so
	 $$
	 	\mu(\wt\phi_{h}) = \mu(\psi^{-1}\wt\phi_{h}\psi)
	 $$
	 for any $h \in C^{\infty}(V)$ and $\psi \in \Cont_{0}(V)$, where we use the natural action of $\Cont_0(V)$ on $\wt\Cont_0(V)$ by conjugation.  Furthermore 	 
	 $\psi^{-1}\wt\phi_{h}\psi = \wt\phi_{g}$
	 where	
	 $$
	 	g:= \alpha(d\psi^{-1}(X_{h})\circ \psi) = (f \cdot h) \circ \psi
	 $$
	 for some positive function $f \in C^{\infty}(V)$, so in particular $g|_{Y} > 0$ if and only if $h|_{\psi(Y)} > 0$.
	 Hence $Y$ is $\mu$-superheavy if and only if $\psi(Y)$ is $\mu$-superheavy and likewise for $\mu$-subheavy.
	 
	 For item (iv) recall that we assume all quasi-morphisms are homogeneous and non-zero.
	Suppose there is an $h \in C^{\infty}(V)$ such that $h > 0$ and $\mu(\wt\phi_{h}) = 0$.
	For any integer $m$ it follows that $\mu(\wt\phi_{mh}) = \mu(\wt\phi_{h}^{m}) = 0$ since
	$h$ is autonomous.  Since for any $k \in C^{\infty}([0,1] \times V)$ there is a positive
	integer $m$ such that $-m h \leq k \leq mh$, it follows from the monotonicity
	of $\mu$ that $\mu(\wt\phi_{k}) = 0$.  Therefore $\mu = 0$, which is a contradiction. 
\end{proof}

\begin{proof}[Proof of Theorem~\ref{t:SHandsubheavy}]
	For item (i), let $Y$ be $\mu$-superheavy and $h$ be an autonomous contact Hamiltonian where $h|_{Y} = 0$.
	 Recall for any $\phi \in \Cont_{0}(V)$ that $\phi^{*}\alpha = k \alpha$ where $k\fc V \to \R$ is a positive function.
	 It follows then for any positive integer $m$ and real number $\epsilon > 0$ that 
	 $$
	 	g_{t}:= \alpha(X_{mh} + d\phi^{t}_{mh}(m\epsilon R_{\alpha}) \circ (\phi_{mh}^{t})^{-1})
	 $$
	 which is the contact Hamiltonian so that $\phi_{g}^{t} = \phi_{mh}^{t}\phi_{m\epsilon}^{t}$ for all $t \in [0,1]$,
	 satisfies $g_{t}|_{Y} > \delta$ for some $\delta > 0$.	
	 Using that $Y$ is $\mu$-superheavy and $\mu$ is monotone we have
	 $$
	 	\mu(\wt\phi_{mh}\wt\phi_{m\epsilon}) = \mu(\wt\phi_{g}) > 0\,.
	 $$
	 Since $h$ is autonomous, $\wt\phi_{mh} = \wt\phi_{h}^{m}$, and using
	 $\mu$ is a homogeneous quasi-morphism \eqref{e:Quasimorphism} we get
	 $$
	 	m\, \mu(\wt\phi_{h}) = \mu(\wt\phi_{mh}) \geq \mu(\wt\phi_{mh}\wt\phi_{m\epsilon}) + \mu(\wt\phi_{m\epsilon}^{-1}) - D > 
		\mu(\wt\phi_{m\epsilon}^{-1}) - D = -m\epsilon\, \mu(\wt\phi_{1}) - D.
	 $$
	 By dividing through by $m$ and taking the limit as $m \to \infty$ gives
	 $\mu(\wt\phi_{h}) > -\epsilon\, \mu(\wt\phi_{1})$
	 for all $\epsilon > 0$, and therefore taking $\epsilon \to 0$ gives
	 $$
	 	\mu(\wt\phi_{h}) \geq 0.
	 $$
	 One proves $\mu(\wt\phi_{h}) \leq 0$ similarly 
	 using Lemma~\ref{l:LessThanZero}.  Therefore $\mu(\wt\phi_{h}) = 0$ and hence $Y$ is $\mu$-subheavy.
	 
	To prove item (ii), suppose that $Y$ and $Z$ are disjoint and pick a contact Hamiltonian $h$ so that $h|_{Y} > 0$
	and $h|_{Z} = 0$, which is possible since $Y,Z$ are closed subsets.  This leads to a contradiction since by the definitions this implies $\mu(\phi_{h}) > 0$ and $\mu(\phi_{h}) = 0$.
\end{proof}

\noindent
Next up is the proof of Theorem~\ref{t:fiberSH} about the existence of nondisplaceable pre-Lagrangians in prequantizations of toric symplectic manifolds.
\begin{proof}[Proof of Theorem~\ref{t:fiberSH}]
	Let $P\fc M^{2n} \to \Delta \subset \R^{n}$ be a moment map for the toric structure on $M$, let 
	$\pi\fc (V, \alpha) \to (M, \w)$ be the prequantization map, and let $\wh{P} = P \circ \pi\fc V \to \Delta$.
	Every fiber of $\wh{P}$ is either a pre-Lagrangian torus or a sits over a strictly isotropic torus in $M$ and the latter
	are always displaceable \cite{Lau86H}, so it suffices to show not every fiber of $\wh{P}$ is displaceable.
	
	Suppose every fiber of $\wh{P}$ is displaceable, then we can take an open cover $\{U_{j}\}_{j=1}^{d}$ 
	 of $\Delta$ such that each $\wh{P}^{-1}(U_{j}) \subset V$ is displaceable.
	Since the coordinate functions of $P$ commute, for any two functions $f,g\fc \R^{n} \to \R$ 
	the contactomorphisms
	 $\wt\phi_{\wh{P}^{*}\!f}$ and  $\wt\phi_{\wh{P}^{*}\!g}$ commute and  $\wt\phi_{\wh{P}^{*}\!(f+g)} =  \wt\phi_{\wh{P}^{*}\!f}
	 \wt\phi_{\wh{P}^{*}\!g}$.    In particular if $\{f_{j}\}$ is a partition of
	 unity subordinate to $\{U_{j}\}$, then
	 $$
	 	\mu(\wt{\phi}_{1}) = \mu(\wt\phi_{\wh{P}^{*}\!f_{1}} + \dots + \wt\phi_{\wh{P}^{*}\!f_{d}}) = \sum_{j=1}^{d}
		\mu(\wt\phi_{\wh{P}^{*}\!f_{j}}) = 0
	 $$
	since homogeneous quasi-morphisms are homomorphisms when restricted to abelian subgroups
	and also that 
	$\mu(\wt\phi_{\wh{P}^{*}\!f_{j}})=0$ by
	 the vanishing property.  However $\mu(\wt\phi_{1}) > 0$, so we have a contradiction.
\end{proof}

\begin{rem}
	The proof of Theorem~\ref{t:fiberSH} also shows if there is monotone quasi-morphism $\mu\fc \wt{\Cont}_0(V, \xi) \to \R$
	with the vanishing property and $(V, \xi)$ is completely integrable contact manifold, in the sense of
	Khesin--Tabachnikov \cite{KheTab10C}, then at least one of the pre-Lagrangian fibers is non-displaceable.
\end{rem}

\noindent
Let us now prove Proposition~\ref{p:subheavyReebImpliesSuperheavy} which states that if a subheavy subset $Y \subset V$ is preserved
by a positive contact vector field, then it is $\mu$-superheavy.

\begin{proof}[Proof of Proposition~\ref{p:subheavyReebImpliesSuperheavy}]
	We will assume that $Y$ is invariant under the
	flow for the Reeb vector field $R_{\alpha}$, since any positive contact vector field is the Reeb vector for some contact form
	\cite[Chapter 3.4]{McDSal98I}.
	Given $h \in C^{\infty}(V)$ such that $h|_{Y} > 0$, since $Y$ is closed we have $h|_{Y} \geq c$ for some positive $c \in \R$.  Let $m$ be a positive integer and 
	note that $\phi_{g}^{t} = \phi_{-mc}^{t}\phi_{mh}^{t}$ where
	$$
	g_{t}:= \alpha(-mc R_{\alpha} + d\phi^{t}_{-mc}(m X_{h}) \circ (\phi_{-mc}^{t})^{-1}) = m(-c + h\circ\phi_{mc}^{t}).
	$$
	Since $\phi_{mc}^{t} = \phi^{t}_{mcR_{\a}}$ is a reparametrization of the Reeb flow, which preserves $Y$, it follows
	that $g_{t}|_{Y} \geq 0$ and hence $\mu(\wt\phi_{-mc}\wt\phi_{mh}) = \mu(\wt\phi_{g}) \geq 0$ since $\mu$ is monotone
	and $Y$ is $\mu$-subheavy.
	 Since $h$ is autonomous it follows $\wt\phi_{mh} = \wt\phi_{h}^{m}$ and because
	 $\mu$ is a homogeneous quasi-morphism we have
	 $$
	 	m\, \mu(\wt\phi_{h}) = \mu(\wt\phi_{mh}) \geq \mu(\wt\phi_{-mc}\wt\phi_{mh}) + \mu(\wt\phi_{mc}) - D \geq 
		\mu(\wt\phi_{mc}) - D = m\, \mu(\wt\phi_{c}) - D.
	 $$
	 By dividing through by $m$ and taking the limit as $m \to \infty$ gives
	 $\mu(\wt\phi_{h}) \geq \mu(\wt\phi_{c})$ and $\mu(\wt\phi_{c}) > 0$ since $V$ is $\mu$-superheavy
	 by Proposition~\ref{p:Basics}(iv).
\end{proof}
%%%%%%%%%%%%%%%%%%%%%%%%%%%%%%%%%%%%%%%%%%%%%%%%%%%%%%%%%%

%%%%%%%%%%%%%%%%%%%%%%%%%%%%%%%%%%%%%%%%%%%%%%%%%%%%%%%%%%
\subsection{Proof of results from Section~\ref{s:SymplecticQMQS}}\label{s:ProofPrequantRigid}

Here we will prove the results in Section~\ref{s:SymplecticQMQS} about the relation between
quasi-morphisms on $\wt\Cont_{0}(V)$ and $\wt\Ham(M)$ when $\pi\fc (V, \alpha) \to (M, \w)$ is a prequantization.
Before proving Theorem~\ref{thm_HamQuasimorphism} we need the following lemma.

\begin{lemma}\label{l:PullApart}
	Let $\pi\fc (V, \alpha) \to (M, \w)$ be a prequantization and let $\mu\fc \wt\Cont_{0}(V) \to \R$ be a
	monotone quasi-morphism, then
	$$
		\mu(\wt{\phi}_{c + \pi^{*}\!H}) = \left(\int_{0}^{1}c(t)dt\right)\mu(\wt{\phi}_{1}) + \mu(\wt{\phi}_{\pi^{*}\!H})
	$$
	for all smooth functions $H\fc [0,1] \times M \to \R$ and $c\fc [0,1] \to \R$.
\end{lemma}
\begin{proof}
	By using the contact Poisson bracket \eqref{e:ContactPB}, or just the definitions, 
	one sees that
	$\wt\phi_{c}$ and $\wt\phi_{\pi^{*}\!H}$ commute in $\wt\Cont_{0}(V)$ and 
	$\wt\phi_{c+\pi^{*}\!H} = \wt\phi_{c}\,\wt\phi_{\pi^{*}\!H}$.
	Therefore since homogeneous quasi-morphisms are homomorphisms on abelian subgroups
	$$
		\mu(\wt{\phi}_{c + \pi^{*}\!H}) = \mu(\wt{\phi}_{c}) + \mu(\wt\phi_{\pi^{*}\!H})
	$$
	and hence it suffices to prove $\mu(\wt{\phi}_{c}) = \left(\int_{0}^{1}c(t)dt\right)\mu(\wt{\phi}_{1})$.
	
	Since $\wt\phi_{\kappa} = \wt\phi_{c}$ via a time-reparametrization where
	$\kappa = \int_{0}^{1} c(t)dt$, this reduces to proving $\mu(\wt\phi_{\kappa}) = \kappa\, \mu(\wt\phi_{1})$ for
	all real numbers $\kappa \in \R$.
	For any integer $m \in \Z$, this holds
	since
	$\mu$ is homogeneous and $\wt\phi_{m} = \wt\phi_{1}^{m}$.  This extends to rational numbers
	and since $\mu$ is monotone it then holds for all real scalars.	
\end{proof}

\begin{proof}[Proof of Theorem~\ref{thm_HamQuasimorphism}]
	Since $\pi^{*}\fc \wt\Ham(M) \to \wt\Cont_{0}(V)$ from \eqref{e:PullBackHamiltonians} is a homomorphism it is clear that
	$\mu_{M}$ is a quasi-morphism.  For stability let
	$c(t):= \min_{M} (H_{t}- G_{t})$, then by monotonicity and Lemma~\ref{l:PullApart} we have
	$$
		\left(\int_{0}^{1}c(t)dt\right)\mu(\wt{\phi}_{1}) + \mu(\wt{\phi}_{\pi^{*}\!G})
		= \mu(\wt{\phi}_{c + \pi^{*}\!G}) \leq \mu(\wt{\phi}_{\pi^{*}\!H})
	$$
	and hence
	$$
		\left(\int_{0}^{1}\min_{M} (H_{t}- G_{t})\,dt\right)\mu(\wt{\phi}_{1}) \leq \mu(\wt{\phi}_{\pi^{*}\!H}) - \mu(\wt{\phi}_{\pi^{*}\!G})\,.
	$$
	After translating to the definition of $\mu_{M}$ in \eqref{e:muM} this is the left-hand part of the stability condition
	\eqref{e:stable}.
	  The
	right-hand side is proved analogously.  
 
	Lemma~\ref{l:PullApart} shows that the formulas for $\zeta_{\mu_{M}}$ in
	\eqref{e:QuasiStateFormula} and Theorem~\ref{thm_HamQuasimorphism} are equal.
	It follows from the formula in Theorem~\ref{thm_HamQuasimorphism} that if $\mu$
	has the vanishing property, then so does $\zeta_{\mu_{M}}$.  This is because if
	$X \subset M$ is displaceable by an element of $\Ham(M)$, then
	$\pi^{-1}(X) \subset V$ is displaceable by an element of $\Cont_{0}(V)$.
	Going back to the quasi-morphism $\mu_{M}$, it follows from \cite[Proposition 1.7]{Bor13Q} that
	$\mu_{M}$ has the Calabi property if the associated quasi-state $\zeta_{\mu_{M}}$ has the vanishing
	property.
\end{proof}

\begin{proof}[Proof of Proposition~\ref{prop_prequant_rigidity}]
	For item (i), it is enough to show that if $H \in C^\infty(M)$ is such that $H|_{\pi(Y)}=0$, then $\zeta_{\mu_{M}}(H) = 0$. 
	If $H|_{\pi(Y)} = 0$, then $\pi^*H|_Y=0$ and $\mu(\wt\phi_{\pi^{*}H}) = 0$ by the definition of $Y$ being
	$\mu$-subheavy.  It then follows from Theorem~\ref{thm_HamQuasimorphism} that $\zeta_{\mu_{M}}(H) = 0$.

	For item (ii), let $Y  =\pi^{-1}(X)$ and let $h \in C^\infty(V)$ be such that $h|_Y > 0$. 
	There is $H \in C^\infty(M)$ with $\pi^*H \leq h$ and $H|_X > 0$. 
	From the monotonicity of $\mu$ and Theorem~\ref{thm_HamQuasimorphism} we have
	$$
	\mu(\wt\phi_h) \geq \mu(\wt\phi_{\pi^*H}) = \mu(\wt\phi_1)\zeta_{\mu_{M}}(H)
	$$
	and therefore we are done since $\zeta_{\mu_{M}}(H) \geq \min_{X} H > 0$ by the definition of $\zeta_{\mu_{M}}$-superheavy
	and since $\mu(\wt\phi_{1}) > 0$ by Proposition~\ref{p:Basics}(iv). 
\end{proof}
%%%%%%%%%%%%%%%%%%%%%%%%%%%%%%%%%%%%%%%%%%%%%%%%%%%%%%%%%%

%%%%%%%%%%%%%%%%%%%%%%%%%%%%%%%%%%%%%%%%%%%%%%%%%%%%%%%%%%
\subsection{Proofs about Givental's quasi-morphism}\label{s:VanishingProof}

%%%%%%%%%%%%%%%%%%%%%%%
\subsubsection{A brief summary of Givental's quasi-morphism}\label{s:discriminant}

Recall that a point $v \in (V, \xi)$ in a contact manifold is a \tb{discriminant point} for a contactomorphism
$\phi \in \Cont(V, \xi)$ if
\begin{equation}\label{e:dis}
		\phi(v) = v \quad\mbox{and}\quad (\phi^{*}\alpha)_{v} = \alpha_{v}
\end{equation}
for some (and hence every) contact form $\alpha$ and  
the \tb{discriminant} of $\Cont_{0}(V, \xi)$ is
$$
	\Sigma(V, \xi) := \{ \phi \in \Cont_{0}(V, \xi) \mid \mbox{$\phi$ has at least one discriminant point}\}.
$$
A $C^\infty$-generic contactomorphism has no discriminant points.
Indeed if $v$ is a discriminant point of $\phi$, then the image of $d\phi_{v} - \id_{T_vV}$ is contained in $\xi_v$
and hence $d\phi_{v} - \id_{T_vV}$ has a nontrivial kernel.   This means
$v$ is a degenerate fixed point and it is a standard fact that $C^\infty$-generic contactomorphisms do not have degenerate fixed points 
(see \cite[Theorem 3.1]{HofSal95F} for a proof in the Hamiltonian case).
In fact any $\phi \in \Cont_{0}(V)$ on the discriminant 
$\Sigma(V)$ can be perturbed off $\Sigma(V)$ via the Reeb flow, but we will not include the proof since this is not 
necessary for what follows.

In \cite{Giv90N} Givental showed how to coorient the discriminant 
$\Sigma \subset \Cont_{0}(\RP^{2d-1})$ using generating functions.
Given a smooth path $\gamma\fc [0,\tau] \to \Cont_{0}(\RP^{2d-1})$ with endpoints not on $\Sigma$,
the coorientation gives a well-defined intersection index between $\gamma$ and $\Sigma$ denoted
$$
	\mu^{G}(\g) \in \Z
$$
which Givental called the \textbf{nonlinear Maslov index}.  From the intersection viewpoint, Givental
specified \cite[Section 9]{Giv90N} conventions so that $\mu^{G}$ is defined for all paths of contactomorphisms.
Alternatively, as noted by Colin--Sandon \cite[Section 7]{ColSan12T}, the nonlinear Maslov index can be defined purely in 
terms of generating families, 
leading to a uniform definition of the nonlinear Maslov index for any smooth path of contactomorphisms of $\RP^{2d-1}$.
Here are some key properties of the nonlinear Maslov index:
\begin{enumerate}
\item Given two paths $\gamma_i: [0, \tau_i] \to \Cont_0(\RP^{2d-1})$ with
$\gamma_0(\tau_0) = \gamma_1(0)$, one has
\begin{equation}\label{e:concat}
	\mu^G(\gamma_0) + \mu^{G}(\gamma_1) = \mu^G(\gamma_0*\gamma_1)
\end{equation}
where $\gamma_0*\gamma_1:[0, \tau_0 + \tau_1] \to \Cont_0(\RP^{2d-1})$ is their concatenation.
\item For any path $\gamma$ in $\Cont_0(\RP^{2d-1})$ and element $\phi \in \Cont_0(\RP^{2d-1})$
\begin{equation}\label{e:QMP}
		\abs{\mu^{G}(\gamma\phi) - \mu^G(\gamma)} \leq 2d
\end{equation}
where $\gamma\phi$ is the path defined by $t \mapsto \gamma(t)\phi$.
\item If a path $\g$ in $\Cont_{0}(\RP^{2d-1})$ is disjoint from the discriminant, then
\begin{equation}\label{e:vanish}
	\mu^{G}(\g) = 0.
\end{equation}
\item The nonlinear Maslov index $\mu^{G}(\g)$ is invariant under homotopies of $\g$ with fixed endpoints.
\end{enumerate}
The first item follows from the construction as an intersection index, the second item is \cite[Theorem 9.1(a)]{Giv90N},
and the final two properties are established by both Givental \cite[Section 9]{Giv90N} and Colin--Sandon \cite[Section 7]{ColSan12T}.

If $\cP\!\Cont_{0}(\RP^{2d-1})$ denotes the space of contact isotopies $\{\phi^{t}\}_{t \in [0,1]}$
with $\phi^{0} = \id$, then one defines the \textbf{asymptotic nonlinear Maslov index} to be
\begin{equation}\label{e:DefGiv}
	\mu_{\Giv}(\{\phi^{t}\}_{t \in [0,1]}) := \lim_{\tau \to \infty} \frac{\mu^{G}(\{\phi^{t}\}_{t \in [0, \tau]})}{\tau}
\end{equation}
where $\{\phi^{t}\}_{t \in [0,\tau]}$ is given by concatenation so $\phi^{k+s}:=\phi^{s}(\phi^{1})^{k}$ for $s \in [0,1]$ and $k \in \N$.
Since $\mu^{G}$ is invariant under homotopies with fixed endpoints, the map in \eqref{e:DefGiv} descends to a map
$$
	\mu_{\Giv}\fc \wt\Cont_{0}(\RP^{2d-1}) \to \R
$$
and this is the definition of Givental's quasi-morphism from \eqref{e:GiventalQM}.
As a special case of \eqref{e:DefGiv} we have 
\begin{equation}\label{e:limitdef}
	\mu_{\Giv}(\wt{\phi}) = \lim_{m\to \infty} \frac{\mu^{G}(\wt{\phi}^m)}{m}
\end{equation}
for $\wt{\phi} \in \wt{\Cont}_0(\RP^{2d-1})$
and hence $\mu_{\Giv}$ is homogeneous: $\mu_{\Giv}(\wt{\phi}^m) = m\, \mu_{\Giv}(\wt{\phi})$.

%%%%%%%%%%%%%%%%%%%%%%%

%%%%%%%%%%%%%%%%%%%%%%%
\subsubsection{A subheavy Legendrian}

\begin{proof}[Proof of Lemma~\ref{exam_Legendrian_RP_subheavy}]
	It suffices to prove
	$\RP^{d-1}_{L} \subset \RP^{2d-1}$ is $\mu_{\text{Giv}}$-subheavy
	since it is preserved by radial projection \eqref{e:RescaleContact}.	
	
	If $h$ is an autonomous contact Hamiltonian that vanishes on $\RP^{d-1}_{L}$,
	then $X_{h}$ is always tangent to $\RP_{L}^{d-1}$ since it is Legendrian.
	Therefore the Legendrian nonlinear Maslov index $\mu(\lambda)$ from \cite[Section 9]{Giv90N}
	of the constant path of Legendrians $\lambda := \{\phi_{h}^{t}(\RP_{L}^{d-1})\}_{t\in [0,\tau]}$ 
	vanishes.  By the definition of $\mu_{\Giv}$ in \eqref{e:DefGiv} and
	\cite[Section 9, Corollary 2]{Giv90N} we know
	$$
		\mu_{\Giv}(\wt\phi_{h}) = \lim_{\tau \to \infty} \frac{\mu^{G}(\{\phi^{t}_{h}\}_{t \in [0,\tau]})}{\tau}
		=  \lim_{\tau \to \infty} \frac{\mu(\{\phi^{t}_{h}(\RP_{L}^{d-1})\}_{t \in [0,\tau]})}{\tau}
	$$
	so $\mu_{\Giv}(\wt\phi_{h}) = 0$ and therefore $\RP^{d-1}_{L}$ is $\mu_{\Giv}$-subheavy.
\end{proof}
%%%%%%%%%%%%%%%%%%%%%%%

%%%%%%%%%%%%%%%%%%%%%%%
\subsubsection{Proving properties of Givental's quasi-morphism in Proposition \ref{prop_ANLMI_has_vanishing}}

\begin{proof}[Proof of Proposition \ref{prop_ANLMI_has_vanishing} (Monotonicity)]
	By \cite[Theorem 9.1(b)]{Giv90N}, or equivalently \cite[Lemma 7.6]{ColSan12T},
	we know $\mu^G(\wt{\phi}) \geq 0$ if $\wt\phi \succeq \id$, so it follows from \eqref{e:limitdef}
	that 
	$$
		0 \leq \mu_{\Giv}(\wt{\phi}) \quad\mbox{if}\quad  \id \preceq \wt\phi.
	$$
	Now if $\wt{\phi} \preceq \wt{\psi}$, then $\id \preceq \wt{\psi}^m \circ \wt{\phi}^{-m}$
	and hence $\mu_{\Giv}(\wt{\psi}^m \circ \wt{\phi}^{-m}) \geq 0$.  Using this and that
	$\mu_{\Giv}$ is a homogeneous quasi-morphism, we get
	$$
		m\, \mu_{\Giv}(\wt\psi) - m\, \mu_{\Giv}(\wt\phi) =
		\mu_{\Giv}(\wt\psi^m) - \mu_{\Giv}(\wt\phi^m) \geq
		\mu_{\Giv}(\wt\psi^m \circ \wt\phi^{-m}) - D \geq -D.
	$$
	Dividing by $m$ and taking the limit $m \to \infty$, gives $\mu_{\Giv}(\wt{\phi}) \leq \mu_{\Giv}(\wt{\psi})$
	and hence $\mu_{\Giv}$ is monotone.
\end{proof}

\begin{proof}[Proof of Proposition \ref{prop_ANLMI_has_vanishing} ($C^0$-continuity)]
	Givental proved in \cite[Corollary 3, Section 9]{Giv90N} that $\mu_{\Giv}$ is $C^0$-continuous for
	time-independent contact Hamiltonians and as he explained to us the proof generalizes
	to time-dependent contact Hamiltonians in the following way.
	
	Suppose we have $C^0$-convergence $h^{(n)} \to h$ of contact Hamiltonians in 
	$C^{\infty}([0,1] \times \RP^{2d-1})$.  For a given $\epsilon > 0$, pick an integer $m > 0$ such that
	$6d/m < \epsilon$ and by \cite[Theorem 9.1(c)]{Giv90N} we know that if $n$ is sufficiently large, then
	$$
		\abs{\mu^G(\{\phi_{h^{(n)}}^t\}_{t \in [0,m]}) - \mu^G(\{\phi_{h}^t\}_{t \in [0,m]})} \leq 2d.
	$$
	By \eqref{e:concat} and \eqref{e:QMP}, for any two integers $m, N > 0$ and $k \in C^{\infty}([0,1] \times \RP^{2d-1})$
	one has
	$$
		\abs{\mu^{G}(\{\phi_k^t\}_{t \in [0,Nm]}) - N\,\mu^G(\{\phi_k^t\}_{t \in [0,m]})} \leq 2dN
	$$
	which when applied to the previous inequality gives that if $n$ is sufficiently large, then
	$$
		\abs{\mu^G(\{\phi_{h^{(n)}}^t\}_{t \in [0,Nm]}) - \mu^G(\{\phi_{h}^t\}_{t \in [0,Nm]})} \leq 6dN.
	$$
	Dividing by $Nm$ and taking the limit as $N \to \infty$ gives
	$$
		\abs{\mu_{\Giv}(\wt\phi_{h^{(n)}}) - \mu_{\Giv}(\wt\phi_h)} \leq \frac{6d}{m} < \epsilon
	$$
	if $n$ is sufficiently large and therefore $\lim_{n \to \infty} \mu_{\Giv}(\wt\phi_{h^{(n)}}) = \mu_{\Giv}(\wt\phi)$.
\end{proof}

\begin{proof}[Proof of Proposition \ref{prop_ANLMI_has_vanishing} (Vanishing property)]
	For an open $U \subset \RP^{2d-1}$ suppose there is a 
	$\psi \in \Cont_0(\R P^{2d-1})$ such that $\psi(U) \cap U = \emptyset$
	and without loss of generality we may assume $\psi$ has no discriminant points.
	By \eqref{e:concat} we know
	$$
		\abs{\mu^{G}(\{\phi_{h}^{t}\psi\}_{t \in [0, \tau]}) - \mu^{G}(\{\phi_{h}^{t}\}_{t \in [0, \tau]})} \leq 2d
	$$
	so if $\mu^{G}(\{\phi_{h}^{t}\psi\}_{t \in [0, \tau]}) = 0$ for all $\tau \geq 0$, then
	it will follow from \eqref{e:DefGiv} that
	$$
	\mu_{\Giv}(\wt\phi_{h}) = \lim_{\tau \to \infty} \frac{\mu^{G}(\{\phi_{h}^{t}\}_{t \in [0,\tau]})}{\tau} = 0\,.
	$$		
	Therefore by \eqref{e:vanish} it remains to prove $\phi_h^t\psi$ has no discriminant points
	for all $t \geq 0$.
	Assume $p$ is a discriminant point for some $\phi_{h}^{t}\psi$. If $p \in U$, then $\psi(p) = (\phi_{h}^{t})^{-1}(p) \in U$ 
	but this contradicts that $\psi(U) \cap \ol{U} = \emptyset$.  If $p \not\in U$, then $\psi(p) = (\phi_{h}^{t})^{-1}(p) = p$ so 
	$p$ is a fixed point of $\psi$ and also a discriminant point of $\psi$, but we assumed they did not exist. 
\end{proof}
%%%%%%%%%%%%%%%%%%%%%%%

%%%%%%%%%%%%%%%%%%%%%%%%%%%%%%%%%%%%%%%%%%%%%%%%%%%%%%%%%%%%%%%%%%%

%%%%%%%%%%%%%%%%%%%%%%%%%%%%%%%%%%%%%%%%%%%%%%%%%%%%%%%%%%%%%%%%%%%

\bibliographystyle{alpha}
\bibliography{ContactRed}

\bigskip

\noindent
\begin{center}
\begin{tabular}{ll}
Matthew Strom Borman & Frol Zapolsky\\
Stanford University &  University of Haifa\\
borman@stanford.edu & frol.zapolsky@gmail.com
\end{tabular}
\end{center}

\end{document}

%% file: BWPoly.pdf_tex
%% Creator: Inkscape inkscape 0.48.2, www.inkscape.org
%% Accompanies image file 'BWPoly.pdf' (pdf, eps, ps)
\begingroup%
  \makeatletter%
  \providecommand\color[2][]{%
    \errmessage{(Inkscape) Color is used for the text in Inkscape, but the package 'color.sty' is not loaded}%
    \renewcommand\color[2][]{}%
  }%
  \providecommand\transparent[1]{%
    \errmessage{(Inkscape) Transparency is used (non-zero) for the text in Inkscape, but the package 'transparent.sty' is not loaded}%
    \renewcommand\transparent[1]{}%
  }%
  \providecommand\rotatebox[2]{#2}%
  \ifx\svgwidth\undefined%
    \setlength{\unitlength}{523.37970921bp}%
    \ifx\svgscale\undefined%
      \relax%
    \else%
      \setlength{\unitlength}{\unitlength * \real{\svgscale}}%
    \fi%
  \else%
    \setlength{\unitlength}{\svgwidth}%
  \fi%
  \global\let\svgwidth\undefined%
  \global\let\svgscale\undefined%
  \makeatother%
  \begin{picture}(1,0.66232644)%
    \put(0,0){\includegraphics[width=\unitlength]{BWPoly.pdf}}%
    \put(0.00185373,0.00432585){\color[rgb]{0,0,0}\makebox(0,0)[lb]{\smash{$\epsilon_1$}}}%
    \put(0.45341036,0.04181353){\color[rgb]{0,0,0}\makebox(0,0)[lb]{\smash{$\epsilon_2$}}}%
    \put(0.20974019,0.64672902){\color[rgb]{0,0,0}\makebox(0,0)[lb]{\smash{$\epsilon_3$}}}%
    \put(0.54286436,0.19456211){\color[rgb]{0,0,0}\makebox(0,0)[lb]{\smash{$\epsilon_1$}}}%
    \put(0.8856023,0.18669631){\color[rgb]{0,0,0}\makebox(0,0)[lb]{\smash{$\epsilon_2$}}}%
    \put(0.80600008,0.64093802){\color[rgb]{0,0,0}\makebox(0,0)[lb]{\smash{$\epsilon_3$}}}%
  \end{picture}%
\endgroup%